\documentclass{amsart}
\usepackage[T1]{fontenc}
\usepackage[utf8]{inputenc}
\usepackage{lmodern}
\usepackage{microtype}
\usepackage[a4paper,margin=30mm]{geometry}
\usepackage{amsmath,amssymb,amsthm,mathtools}
\usepackage{enumitem}
\usepackage{xcolor}
\usepackage[numbers,sort&compress]{natbib}
\usepackage{aliascnt}
\usepackage[colorlinks=true,linkcolor=blue!55!black,citecolor=blue!55!black,urlcolor=blue!55!black]{hyperref}

\hypersetup{
  pdftitle={Cycle-Decorated Ribbon Complexes: Cut Coproducts and Alternating-Fence Positivity},
  pdfauthor={Pyuyi Chufeng Huang},
  pdfkeywords={noncommutative symmetric functions, ribbon Schur functions, equivariant homology, combinatorial species, cut coproduct, reflection length, alternating fences}
}

\newtheorem{theorem}{Theorem}[section]
\newaliascnt{lemma}{theorem}
\newtheorem{lemma}[lemma]{Lemma}
\aliascntresetthe{lemma}
\newaliascnt{proposition}{theorem}
\newtheorem{proposition}[proposition]{Proposition}
\aliascntresetthe{proposition}
\newaliascnt{corollary}{theorem}
\newtheorem{corollary}[corollary]{Corollary}
\aliascntresetthe{corollary}
\theoremstyle{definition}
\newaliascnt{definition}{theorem}
\newtheorem{definition}[definition]{Definition}
\aliascntresetthe{definition}
\newaliascnt{example}{theorem}
\newtheorem{example}[example]{Example}
\aliascntresetthe{example}
\theoremstyle{remark}
\newaliascnt{remark}{theorem}
\newtheorem{remark}[remark]{Remark}
\aliascntresetthe{remark}

\usepackage[nameinlink,capitalise,noabbrev]{cleveref}

\newcommand{\kk}{\mathbb{K}}
\newcommand{\QQ}{\mathbb{Q}}
\newcommand{\NSym}{\mathrm{NSym}}
\newcommand{\Sym}{\mathrm{Sym}}
\newcommand{\Des}{\operatorname{Des}}
\newcommand{\id}{\operatorname{id}}

\title[Cycle-Decorated Ribbon Complexes]
{Cycle-Decorated Ribbon Complexes:\\
Cut Coproducts and Alternating-Fence Positivity}
\author{Pyuyi Chufeng Huang}
\address{School of Mathematics, Sichuan University, Chengdu 610064, China}
\email{pyuyi233@gmail.com}
\date{}

\begin{document}
\raggedbottom

\begin{abstract}
Let $\alpha\models n$.  We define a two-variable specialization
$Z_\alpha(t,q)$ of the ribbon basis of noncommutative symmetric functions
from cycle enumerators of an ordinary permutation and a rooted permutation,
with $q$ recording reflection length.  We realize $n!Z_\alpha(t,q)$ as the
shifted bigraded Euler characteristic of an $\mathfrak S_n$-equivariant
ordered-set-partition complex.  When $\alpha$ has at most one odd part,
each total decoration determines a set of simultaneous factorization cuts,
and its fiber is the classical ribbon complex indexed by that cut set.  This
gives explicit nonnegative ribbon expansions for every bigraded homology
representation.

Organizing the complexes on labelled finite sets yields a counital differential
graded comonoid in the Cauchy monoidal category of species, whose cut
coproduct is compatible with the fiber decomposition.  At a factorization
cut, the induced map on top homology is injective; its cokernel has the
near-concatenation ribbon character, and the kernel of the aggregate reduced
cut coproduct on homology is $H_1$.

For the zigzag compositions $\delta_n$, $Z_{\delta_n}(t,-1)$ is the order
polynomial of the alternating fence.  We prove
\[
 n!Z_{\delta_n}(t,-u)\in\mathbb N[t,u]
\]
using coefficientwise nonnegative recurrences derived from a Riccati
equation.  The same cut data define a nonnegative factorization defect
$d=\lceil n/2\rceil-k-j$, which controls the homological support and makes
the Euler sign constant on each defect layer.  We also determine the full
defect-zero edge by explicit Frobenius-character formulas.
\end{abstract}

\maketitle

\section{Introduction}
\label{sec:introduction}

Ribbon functions connect Boolean inclusion--exclusion, descent
representations, and rank-selected homology.  For a nonempty composition
$\alpha\models n$, the ribbon element in the algebra of noncommutative
symmetric functions is
\[
 R_\alpha=\sum_{\beta\succeq\alpha}
 (-1)^{\ell(\alpha)-\ell(\beta)}H_\beta.
\]
Under abelianization, $R_\alpha$ maps to the ribbon Schur function
$r_\alpha$, which is the Frobenius characteristic of the corresponding
descent representation.  We refine this picture by two cycle statistics and
then study the zigzag compositions that encode alternating fences.

For $m\geq0$, write
\[
 m^+=\left\lceil\frac m2\right\rceil,
 \qquad m^-=\left\lfloor\frac m2\right\rfloor,
\]
and set
\[
 \mathsf h_m(t,q)
 =\frac1{m!}
   \prod_{i=1}^{m^-}(t+i)
   \prod_{j=0}^{m^+-1}(t+jq).
\]
The second product is the cycle enumerator of a permutation in
$\mathfrak S_{m^+}$, with $q$ recording its reflection length
$m^+-c(\sigma)$.  The first product is obtained from the cycle enumerator of a
permutation on $m^-+1$ letters with one distinguished letter.  Since the
complete functions freely generate $\NSym$, the assignments
$H_m\mapsto\mathsf h_m(t,q)$ define an algebra homomorphism $\phi_{t,q}$.
We write
\[
 Z_\alpha(t,q)=\phi_{t,q}(R_\alpha).
\]

\subsection*{Main results}

Our first result gives an equivariant homological realization of the ribbon
evaluation.  The chain groups are spanned by ordered set partitions whose
blocks carry an ordinary permutation and a rooted permutation.  Adjacent
blocks are merged by direct sum in the ordinary component and by gluing the
two root cycles in the rooted component.  Suppose that $\alpha$ has at most
one odd part.  For a total decoration $\theta$, let $F_\alpha(\theta)$ be the
set of cuts at which both components factor.  The fiber with total decoration
$\theta$ is the classical ribbon complex with cut set $F_\alpha(\theta)$.
Consequently, \cref{thm:total-decoration-ribbon-homology} gives
\begin{equation}
 \operatorname{ch}_{t,q}H_k(C_\bullet^{t,q}(\alpha))
 =\sum_{\substack{\theta\in\mathcal D_n\\
                   |F_\alpha(\theta)|=k-1}}
 t^{\deg_t\theta}q^{\deg_q\theta}r_{\gamma_\alpha(\theta)}.
 \label{eq:intro-homology-formula}
\end{equation}
Thus every bigraded homology representation is determined explicitly and is
Schur-positive.  Grouping by exact cut set gives the nonnegative multiplicity
polynomials in \cref{cor:exact-cut-decomposition}.

The same factorization data define a cut coproduct.  After the complexes are
organized as a species on finite label sets, the coproduct makes their direct
sum a counital differential graded comonoid in the Cauchy monoidal category.
It respects the total-decoration fibers.  At every factorization cut, the
aggregate label-set component induces a canonical injection of the
corresponding ribbon homology module into the induction product of the two
factor modules; the cokernel has the near-concatenation ribbon character.
The kernel on homology of the aggregate reduced coproduct is exactly $H_1$;
see
\cref{thm:total-decoration-cut-coproduct,cor:ribbon-cut-exact-sequence,cor:cut-primitive-homology}.

Our second result concerns the zigzag compositions
$\delta_{2N}=(2,\ldots,2)$ and
$\delta_{2N+1}=(2,\ldots,2,1)$.  At $q=-1$, their evaluations are the order
polynomials of alternating fences.  Retaining reflection length by writing
$q=-u$, \cref{thm:cycle-sign-coherence} proves
\begin{equation}
 n!Z_{\delta_n}(t,-u)\in\mathbb N[t,u].
 \label{eq:intro-bivariate-positivity}
\end{equation}
The proof starts from the hypergeometric differential equation for the even
kernel and converts it into two simultaneous recurrences with nonnegative
coefficients.  The odd case follows from a first-order relation between the
even and odd kernels.

The same total decoration controls the two parts of the paper.  Its
factorization cuts select the ribbon complex occurring in homology, while the
excess of ordinary cycles over the factorization intervals determines the
sign after the substitution $q=-u$.  More precisely, if
$\theta=(\sigma,\tau)$ contributes in homological degree $k$ and
$q$-degree $j$, then
\[
 \operatorname{fd}_n(\theta)
 =c(\sigma)-|F_{\delta_n}(\theta)|-1
 =\left\lceil\frac n2\right\rceil-k-j\geq0.
\]
The differential preserves total decoration and therefore splits the zigzag
complex by this factorization defect.  On each defect layer the shifted Euler
sign is constant, and \cref{thm:factorization-defect-layers} gives the support
condition
\[
 H_k^{i,j}(n)\neq0
 \quad\Longrightarrow\quad
 k+j\leq\left\lceil\frac n2\right\rceil\leq i+j.
\]
The defect-zero edge of this region is computed in
\cref{cor:defect-zero-edge} by explicit representation-valued formulas.  The
virtual Euler characteristic is not Schur-positive in general, already for
$n=4$, so the scalar positivity in \eqref{eq:intro-bivariate-positivity} does
not arise from coefficientwise positivity in the representation ring.

The Riccati recurrence also gives positive integral logarithmic coefficients
for the even zigzag series.  At $u=1$, they specialize to scaled order
polynomials of alternating circular fences; see
\cref{prop:connected-circular-deformation}.

\subsection*{Relation to earlier work}

Up to augmentation and homological reindexing, the undecorated complex is the
ordered-set-partition complex $\Delta_\alpha$ of Ehrenborg and Jung.  They
identify it with a rank-selected Boolean complex, prove concentration of its
homology, construct a basis indexed by exact descent composition, and identify
the top homology with the border-strip Specht module
\cite[Theorem~4.2, Theorem~6.6, and Proposition~7.3]{EhrenborgJung2013}.
Bergeron and Krob constructed acyclic complexes related to ribbon elements in
noncommutative symmetric functions \cite{BergeronKrob1997}.  After choosing an
$n$-dimensional vector space with an ordered basis, the multidegree
$(1,\ldots,1)$ part of VandeBogert's refinement complex for its symmetric
algebra has the same ordered-set-partition basis and adjacent multiplication
maps, up to homological reindexing
\cite[Definition~A.7 and Example~A.8]{KellerVandeBogert2025}.  In the present
complex, these classical ribbon complexes occur as fibers indexed by total
permutation decorations, and the simultaneous factorization cuts determine
both the fiber and its bidegree.

Chain-level categorifications of the ribbon concatenation and
near-concatenation identity occur in VandeBogert's ribbon Schur functors and,
for projective $0$-Hecke modules, in the canonical split exact sequences of
Almousa and Lu \cite{KellerVandeBogert2025,AlmousaLu2026}.  Here the injection
in \cref{cor:ribbon-cut-exact-sequence} is induced by a coassociative label-set
cut coproduct on $\mathfrak S_n$-equivariant complexes.  It restricts to each
total-decoration fiber and is compatible with simultaneous factorization at
every cut.  Hersh and Sundaram construct ribbon bases for
rank-selected homology and Whitney homology of geometric lattices
\cite[Theorems~5.29--5.30]{HershSundaram2026}.

Gessel and Reutenauer jointly enumerate cycle structure and descent sets of a
single permutation \cite{GesselReutenauer1993}.  Novelli, Thibon, and Toumazet
lift the cycle-index polynomials of symmetric groups to bases of
quasisymmetric and noncommutative symmetric functions and obtain a product
formula and a combinatorial-complex recurrence
\cite{NovelliThibonToumazet2020}.  In our construction, cycle statistics are
carried by the block decorations, while the descent set is supplied by the
homology of the ribbon fiber.  Direct-sum factorization of ordinary
permutations is standard; see
\cite[Equation~(34), Proposition~3.3, and the discussion before
Theorem~4.7]{HivertNovelliThibon2008}.  The construction below additionally uses a rooted component and requires
simultaneous factorization of the two components at the same cuts.

Kreweras's determinant for skew plane partitions \cite{Kreweras1965}, in the
form used by Ferroni, Morales, and Panova, gives the alternating-fence
specialization below.  Ferroni, Morales, and Panova prove coefficientwise
positivity for fence and circular-fence order polynomials
\cite{FerroniMoralesPanova}, and Kahane gives a permutation statistic for the
coefficients of every fence order polynomial
\cite[Theorem~3.8]{KahaneFenceCoefficients}.  The author's earlier preprint
gives a different record model for the same one-variable specialization and
for circular fences \cite{HuangGreedyRecords2026}.  These results concern
$u=1$; the refinement here retains reflection length and relates its sign to
the homological factorization defect.

\subsection*{Organization and conventions}

\Cref{sec:characters} defines the two-variable ribbon evaluation.
\Cref{sec:decorated-complex} constructs the decorated complex and computes its
equivariant homology.  \Cref{sec:staircase-consequences} proves the
alternating-fence specialization, bivariate positivity, the defect
decomposition, the support and defect-zero formulas, and the logarithmic
consequence.

All vector spaces are over a field $\kk$ of characteristic zero, and
$\mathbb N$ includes zero.  For $m\geq0$, write $[m]=\{1,\ldots,m\}$, with
$[0]:=\varnothing$; an integer interval $\{a,\ldots,b\}$ is empty when
$a>b$.  Polynomial characters take values in $\QQ[t,q]$.  Chain complexes
are homologically graded.  The superscript $t,q$ records the two weight
gradings; the differential does not depend on numerical values of $t$ and
$q$.  Compositions in \cref{sec:decorated-complex} are nonempty.  Empty
internal alphabets and cut sets are allowed.  Throughout, $\mathfrak S_n$
acts on the labels of ordered set partitions and fixes the decoration bases.

\section{Ribbon evaluations weighted by cycle statistics}
\label{sec:characters}

The ribbon basis turns coarsenings into inclusion--exclusion.  We evaluate
the complete generators by two cycle enumerators.

\subsection{Compositions, ribbons, and \texorpdfstring{\(\mathrm{NSym}\)}{NSym}}

A \emph{composition} of a positive integer \(n\), denoted
\(\alpha\models n\), is a sequence
\(\alpha=(\alpha_1,\ldots,\alpha_k)\) of positive integers with
\[
 |\alpha|:=\alpha_1+\cdots+\alpha_k=n.
\]
Its length is \(\ell(\alpha):=k\).  We also use the empty composition
\(\varnothing\models0\), of length zero.

A composition \(\gamma\) is a \emph{coarsening} of \(\alpha\) if it can be
obtained from \(\alpha\) by repeatedly replacing adjacent parts by their sum.
Write \(\operatorname{Coars}(\alpha)\) for all coarsenings of \(\alpha\),
including \(\alpha\).

Let \(\mathcal R:=\mathbb Q[t,q]\).  The algebra of noncommutative symmetric
functions over \(\mathcal R\) is the free associative algebra
\[
 \mathrm{NSym}_{\mathcal R}=\mathcal R\langle H_1,H_2,\ldots\rangle,
\]
where \(H_0:=1\) and \(\deg H_m=m\).
For \(\alpha=(\alpha_1,\ldots,\alpha_k)\), set
\(H_\alpha=H_{\alpha_1}\cdots H_{\alpha_k}\), with
\(H_\varnothing=1\).  We use the ribbon normalization
\begin{equation}
 R_\alpha
 =\sum_{\gamma\in\operatorname{Coars}(\alpha)}
   (-1)^{\ell(\alpha)-\ell(\gamma)}H_\gamma,
 \qquad R_\varnothing=1.
 \label{eq:ribbon-definition}
\end{equation}
These are standard conventions; see~\cite{GelfandEtAlNSym}.
Under the canonical abelianization $\NSym\to\Sym$, $R_\alpha$ maps to the
ribbon Schur function $r_\alpha$.

\subsection{The two-variable character}

For \(m\geq0\), put
\[
 m^+=\left\lceil\frac m2\right\rceil,
 \qquad
 m^-=\left\lfloor\frac m2\right\rfloor.
\]
Thus \(m^++m^-=m\).  Define
\begin{equation}
 \mathsf h_m(t,q)
 =\frac{1}{m!}
   \left(\prod_{i=1}^{m^-}(t+i)\right)
   \left(\prod_{j=0}^{m^+-1}(t+jq)\right),
 \qquad
 \mathsf h_0(t,q)=1.
 \label{eq:hm-tq}
\end{equation}

\begin{definition}[Two-variable character]
\label[definition]{def:cycle-q-character}
Let
\(\phi_{t,q}:\mathrm{NSym}_{\mathcal R}\rightarrow\mathcal R\) be the unique
unital $\mathcal R$-algebra
homomorphism determined by
\begin{equation*}
 \phi_{t,q}(H_m)=\mathsf h_m(t,q)
 \qquad(m\geq1).
\end{equation*}
For a composition \(\alpha\), define
\begin{equation}
 Z_\alpha(t,q)=\phi_{t,q}(R_\alpha),
 \qquad Z_\varnothing(t,q)=1.
 \label{eq:Z-alpha}
\end{equation}
\end{definition}

The map exists uniquely because \(\mathrm{NSym}_{\mathcal R}\) is freely generated by
the \(H_m\).  Here a character is a unital algebra homomorphism, as usual in
combinatorial Hopf algebra theory~\cite{AguiarBergeronSottile}.

The unsigned Stirling cycle enumerator
\cite[Section~1.3]{StanleyEC1}, after homogenization, gives for $r\geq1$
\begin{equation}
 \sum_{\sigma\in\mathfrak S_r}
 q^{\,r-c(\sigma)}t^{c(\sigma)}
 =\prod_{j=0}^{r-1}(t+jq),
 \label{eq:cycle-expansion}
\end{equation}
where \(c(\sigma)\) counts all cycles, including one-cycles.  Consequently,
\begin{equation*}
 m!\,\phi_{t,q}(H_m)
 =\left(\prod_{i=1}^{m^-}(t+i)\right)
  \sum_{\sigma\in\mathfrak S_{m^+}}
  q^{\,m^+-c(\sigma)}t^{c(\sigma)}.
\end{equation*}

Thus both factors in the definition have direct cycle interpretations, and
$r-c(\sigma)$ is the reflection length of $\sigma$ with respect to all
transpositions.

At \(q=-1\), the two products in~\eqref{eq:hm-tq} combine to give
the falling factorial $(t+m^-)^{\underline m}$, where
$x^{\underline m}=x(x-1)\cdots(x-m+1)$.  Hence
\begin{equation*}
 \phi_{t,-1}(H_m)
 =\binom{t+m^-}{m}
 =\binom{t+\lfloor m/2\rfloor}{m}.
\end{equation*}
Thus, writing \(\phi_t=\phi_{t,-1}\),
\begin{align}
 \phi_t(H_{2a})&=\binom{t+a}{2a},
 \label{eq:even-kernel}\\
 \phi_t(H_{2a+1})&=\binom{t+a}{2a+1}.
 \label{eq:odd-kernel}
\end{align}

Applying \(\phi_{t,q}\) to~\eqref{eq:ribbon-definition} gives
\begin{equation}
 Z_\alpha(t,q)
 =\sum_{\gamma\in\operatorname{Coars}(\alpha)}
  (-1)^{\ell(\alpha)-\ell(\gamma)}
  \prod_{j=1}^{\ell(\gamma)}\mathsf h_{\gamma_j}(t,q).
 \label{eq:Z-coarsening}
\end{equation}
Reversal bijects the coarsenings of \(\alpha\) with those of
\(\alpha^{\mathrm{rev}}\), preserves their lengths, and reverses their parts.
Since the scalar factors in~\eqref{eq:Z-coarsening} commute,
\begin{equation}
 Z_{\alpha^{\mathrm{rev}}}(t,q)=Z_\alpha(t,q).
 \label{eq:reversal-invariance}
\end{equation}

For later use, define the zigzag compositions
\begin{equation*}
 \delta_0=\varnothing,\qquad
 \delta_{2N}=(\underbrace{2,\ldots,2}_{N\text{ parts}}),\qquad
 \delta_{2N+1}=(\underbrace{2,\ldots,2}_{N\text{ parts}},1).
\end{equation*}
By \eqref{eq:reversal-invariance}, the last composition may equivalently be
written \((1,2,\ldots,2)\).

\section{The decorated ribbon complex}
\label{sec:decorated-complex}

Ordered set partitions realize the alternating sum over coarsenings at chain
level.  After recalling the undecorated ribbon complex, we establish the
factorization criterion and apply it to rooted permutation decorations.

Throughout this section, $\alpha\models n$ is nonempty, so $n\geq1$.
All representations are finite-dimensional.

\subsection{The classical ribbon complex}

\begin{definition}[Cut set and coarsening]
\label[definition]{def:cut-set-coarsening}
For a nonempty composition
$\alpha=(\alpha_1,\ldots,\alpha_r)\models n$, define its cut set by
\[
  D(\alpha):=
  \{\alpha_1,\alpha_1+\alpha_2,\ldots,
    \alpha_1+\cdots+\alpha_{r-1}\}\subseteq[n-1].
\]
For the coarsening relation defined in \cref{sec:characters}, write
$\beta\succeq\alpha$.  Equivalently,
\[
  \beta\succeq\alpha
  \quad\Longleftrightarrow\quad
  D(\beta)\subseteq D(\alpha).
\]
Write
\[
  \operatorname{Coars}_k(\alpha)
  :=\{\beta\succeq\alpha:\ell(\beta)=k\}.
\]
\end{definition}

Thus the interval of coarsenings of $\alpha$ is canonically a Boolean lattice
on the $r-1$ cuts in $D(\alpha)$.  If
$\beta=(\beta_1,\ldots,\beta_k)\succeq\alpha$ and $1\leq i<k$, let
\[
  m_i\beta
  :=(\beta_1,\ldots,\beta_{i-1},
      \beta_i+\beta_{i+1},
      \beta_{i+2},\ldots,\beta_k).
\]
Deleting one more cut shows that $m_i\beta\succeq\alpha$.

For $S\subseteq D(\alpha)$, let $\gamma_S$ denote the unique composition of
$n$ with $D(\gamma_S)=S$.

\paragraph{Ordered set partitions.}

\begin{definition}[Ordered set partitions]
\label[definition]{def:ordered-set-partitions}
Let $I$ be a finite set with $|I|=n$, and let
$\beta=(\beta_1,\ldots,\beta_k)\models n$.  Define
\[
  \operatorname{OSP}_{\beta}[I]
  :=\left\{
     (B_1,\ldots,B_k):
     \begin{array}{l}
       B_i\neq\varnothing,\ B_i\cap B_j=\varnothing\ (i\neq j),\\
       B_1\sqcup\cdots\sqcup B_k=I,\ |B_i|=\beta_i
     \end{array}
   \right\}.
\]
For $I=[n]$, let
\[
  M^{\beta}:=\kk[\operatorname{OSP}_{\beta}[n]]
\]
be the permutation module obtained by linearly extending the natural action of
$\mathfrak S_n$ on labels.
\end{definition}

The stabilizer of the standard ordered partition is the corresponding Young
subgroup.  Hence, for every $\beta\models n$,
\begin{align}
  |\operatorname{OSP}_{\beta}[n]|
  &=\binom{n}{\beta_1,\ldots,\beta_k}
    =\frac{n!}{\beta_1!\cdots\beta_k!},\notag\\
  M^{\beta}
  &\cong\operatorname{Ind}_{\mathfrak S_{\beta_1}\times\cdots\times
  \mathfrak S_{\beta_k}}^{\mathfrak S_n}\mathbf{1},
  \qquad
  \operatorname{ch}(M^{\beta})
  =h_{\beta}:=h_{\beta_1}\cdots h_{\beta_k}.
 \label{eq:ordered-partition-module}
\end{align}

\paragraph{The chain complex.}

For $\mathbf B=(B_1,\ldots,B_k)\in
\operatorname{OSP}_{\beta}[n]$, define
\[
  d_i\mathbf B
  :=(B_1,\ldots,B_{i-1},B_i\sqcup B_{i+1},
      B_{i+2},\ldots,B_k)
  \in\operatorname{OSP}_{m_i\beta}[n].
\]
Linear extension gives an $\mathfrak S_n$-equivariant map
\[
  d_i^{\beta}:M^{\beta}\longrightarrow M^{m_i\beta}.
\]

\begin{definition}[Ribbon complex]
\label[definition]{def:ribbon-complex}
For a nonempty composition $\alpha\models n$ with $r=\ell(\alpha)$, set
\[
  C_k(\alpha)
  :=\bigoplus_{\beta\in\operatorname{Coars}_k(\alpha)}M^{\beta},
  \qquad 1\leq k\leq r,
\]
and set $C_k(\alpha)=0$ outside this range.  On the summand $M^{\beta}$,
define
\begin{equation}
\label{eq:ribbon-boundary}
  \partial_k\big|_{M^{\beta}}
  :=\sum_{i=1}^{k-1}(-1)^{i-1}d_i^{\beta}.
\end{equation}
We call $(C_{\bullet}(\alpha),\partial)$ the \emph{ribbon complex} of
$\alpha$, and we place $C_k(\alpha)$ in homological degree $k$.  For Euler
characteristics we use the globally shifted sign $(-1)^{r-k}$, which agrees
with the usual ribbon expansion.
\end{definition}

Associativity of disjoint union gives the usual face identities
\[
 d_i d_j=d_{j-1}d_i\qquad(1\leq i<j\leq k-1).
\]
The two orders of deleting any pair of cuts therefore give equal maps with
opposite signs in \eqref{eq:ribbon-boundary}.  Hence $\partial^2=0$.

\paragraph{Homology.}

Let $\mathcal B_n$ be the Boolean lattice of subsets of $[n]$.  For
$D\subseteq[n-1]$, write
\[
  \mathcal B_n(D):=\{S\subseteq[n]: |S|\in D\}
\]
for its proper rank-selected subposet.

\begin{proposition}[Known ribbon-complex homology]
\label{prop:classical-rank-selected-homology}
Let $\alpha\models n$ be nonempty.  The map
\[
  (B_1,\ldots,B_k)
  \longmapsto
  B_1\subset B_1\sqcup B_2\subset\cdots\subset
  B_1\sqcup\cdots\sqcup B_{k-1}
\]
identifies $C_\bullet(\alpha)$, after the degree shift
$C_k\leftrightarrow\widetilde C_{k-2}$, with the augmented simplicial chain
complex of $\Delta(\mathcal B_n(D(\alpha)))$.  The identification is
$\mathfrak S_n$-equivariant.  The homology is concentrated in degree
$\ell(\alpha)$, and
\begin{align}
 \operatorname{ch}H_{\ell(\alpha)}(C_\bullet(\alpha))&=r_\alpha,\notag\\
 \dim H_{\ell(\alpha)}(C_\bullet(\alpha))
 &=\#\{\pi\in\mathfrak S_n:\Des(\pi)=D(\alpha)\}.
 \label{eq:classical-descent-homology}
\end{align}
\end{proposition}

\begin{proof}
The inverse sends a chain
$S_1\subset\cdots\subset S_{k-1}$ to the successive differences
$S_1,S_2\setminus S_1,\ldots,[n]\setminus S_{k-1}$.  Merging adjacent blocks
deletes the corresponding rank, so this is an equivariant chain
identification; when $D(\alpha)=\varnothing$, we use
$\widetilde C_{-1}(\varnothing)=\kk$.  The homology concentration, descent
basis, and ribbon character follow from
\cite[Theorem~4.2, Theorem~6.6, and Proposition~7.3]{EhrenborgJung2013}; see
also \cite[Theorem~4.3]{Stanley1982GroupsPosets} and
\cite[Theorem~3.4.4]{Wachs2007}.
\end{proof}

\subsection{Factorization along cuts}
\label{subsec:abstract-factorization}

We now give a criterion that decomposes a decorated ribbon complex into
subcomplexes with fixed total decoration.  Fix a nonempty composition
$\alpha\models n$.  For each integer $m$ that occurs as a block size in a
coarsening of $\alpha$, choose a finite set $\mathcal A_m$.  Equip these sets
with a bidegree
\[
 \operatorname{bideg}\colon\mathcal A_m\longrightarrow\mathbb N^2,
\]
whose coordinates we denote by $\deg_t$ and $\deg_q$.  Suppose that whenever
$a,b$ are
adjacent block sizes in such a coarsening there is a product
\[
  \diamond\colon\mathcal A_a\times\mathcal A_b
 \longrightarrow\mathcal A_{a+b}.
\]
We require
\[
 \operatorname{bideg}(x\diamond y)
 =\operatorname{bideg}(x)+\operatorname{bideg}(y)
\]
and
\((x\diamond y)\diamond z=x\diamond(y\diamond z)\) whenever the three
consecutive block sizes occur in a coarsening of $\alpha$.  Thus for every
$\beta=(\beta_1,\ldots,\beta_k)\succeq\alpha$ there is an unambiguous total
product
\begin{equation*}
 \mu_\beta\colon
 \mathcal A_{\beta_1}\times\cdots\times\mathcal A_{\beta_k}
 \longrightarrow\mathcal A_n,
\end{equation*}
where $\mu_{(n)}=\id_{\mathcal A_n}$.

Decorate an ordered set partition of type $\beta$ by one element of
$\mathcal A_{\beta_i}$ on each block.  Merging adjacent blocks and multiplying
their decorations defines a complex $C_\bullet^{\mathcal A}(\alpha)$ with
\begin{equation*}
 C_k^{\mathcal A}(\alpha)
 =\bigoplus_{\substack{\beta\succeq\alpha\\\ell(\beta)=k}}
 \kk[\operatorname{OSP}_\beta[n]]\otimes
 \kk[\mathcal A_{\beta_1}]\otimes\cdots\otimes
 \kk[\mathcal A_{\beta_k}],
\end{equation*}
with the signs of \eqref{eq:ribbon-boundary}.  The group $\mathfrak S_n$
acts on the labels of the ordered set partition and fixes all decorations.
Associativity gives the face identities and hence makes the boundary square
to zero.

\begin{definition}[Factorization condition]
\label[definition]{def:cut-factorizing-system}
The preceding decoration system has \emph{unique factorization along cuts
with respect to $\alpha$} if every map $\mu_\beta$ is injective and, for each
$\theta\in\mathcal A_n$, there is a set
$F_\alpha(\theta)\subseteq D(\alpha)$ such that
\begin{equation}
 \theta\in\operatorname{im}\mu_\beta
 \quad\Longleftrightarrow\quad
 D(\beta)\subseteq F_\alpha(\theta)
 \qquad(\beta\succeq\alpha).
 \label{eq:abstract-cut-factorization-condition}
\end{equation}
Write $\gamma_\alpha(\theta)$ for the composition of $n$ with cut set
$F_\alpha(\theta)$.
\end{definition}

Condition \eqref{eq:abstract-cut-factorization-condition} states that the
factors at all selected cuts exist simultaneously and are unique.

For a finite-dimensional bigraded $\mathfrak S_n$-module $V$, write
\[
 \operatorname{ch}_{t,q}V
 :=\sum_{d,e}\operatorname{ch}(V_{d,e})t^dq^e.
\]
For a finite-dimensional bigraded vector space $V$, write
\[
 \operatorname{Hilb}_{t,q}V:=\sum_{d,e}\dim(V_{d,e})t^dq^e.
\]

\begin{lemma}[Fiber decomposition by total decorations]
\label[lemma]{lem:factorization-fiber-decomposition}
If $\mathcal A$ has unique factorization along cuts with respect to $\alpha$,
then the
total product gives
a canonical bidegree-preserving, $\mathfrak S_n$-equivariant chain
isomorphism
\begin{equation}
 C_\bullet^{\mathcal A}(\alpha)
 \cong
 \bigoplus_{\theta\in\mathcal A_n}
 C_\bullet(\gamma_\alpha(\theta)).
 \label{eq:abstract-factorization-decomposition}
\end{equation}
The summand indexed by $\theta$ is placed in bidegree
$\operatorname{bideg}(\theta)$, and
\begin{equation}
 \operatorname{ch}_{t,q}H_k(C_\bullet^{\mathcal A}(\alpha))
 =\sum_{\substack{\theta\in\mathcal A_n\\
                   |F_\alpha(\theta)|=k-1}}
   t^{\deg_t\theta}q^{\deg_q\theta}
   r_{\gamma_\alpha(\theta)}.
 \label{eq:abstract-factorization-homology}
\end{equation}
 In particular, every bigraded homology Frobenius characteristic admits an
 expansion with coefficients in $\mathbb N[t,q]$ in ribbon Schur functions.
\end{lemma}

\begin{proof}
Associativity implies that every face map preserves the total product.  On a
basis element of type $\beta$, define
\[
 \Psi(\mathbf B;x_1,\ldots,x_k)=\mathbf B
\]
in the summand indexed by
$\theta=\mu_\beta(x_1,\ldots,x_k)$.  Fixing $\theta$, injectivity of
$\mu_\beta$ gives at most one tuple $(x_1,\ldots,x_k)$, and
\eqref{eq:abstract-cut-factorization-condition} says that this tuple exists
exactly when $D(\beta)\subseteq F_\alpha(\theta)$.  These $\beta$ are
precisely the coarsenings of $\gamma_\alpha(\theta)$, so restriction of
$\Psi$ to the subcomplex with total decoration $\theta$ is a basis bijection onto
$C_\bullet(\gamma_\alpha(\theta))$.

Merging adjacent blocks replaces $(x_i,x_{i+1})$ by
$x_i\diamond x_{i+1}$ and does not change $\theta$; under $\Psi$ it is
therefore exactly the ordinary adjacent-union face.  Thus $\Psi$ is a chain
isomorphism.  Degree additivity and the fact that $\mathfrak S_n$ acts only
on the block labels prove the grading and equivariance assertions.
Summing the fibers proves \eqref{eq:abstract-factorization-decomposition},
and \eqref{eq:abstract-factorization-homology} follows from
\cref{prop:classical-rank-selected-homology}.
\end{proof}

\subsection{Permutation decorations}

For \(\sigma\in\mathfrak S_a\) and \(\tau\in\mathfrak S_b\), define
their direct sum \(\sigma\oplus\tau\in\mathfrak S_{a+b}\) by
\[
  (\sigma\oplus\tau)(j)=
  \begin{cases}
    \sigma(j),&1\leq j\leq a,\\
    a+\tau(j-a),&a<j\leq a+b.
  \end{cases}
\]
Direct inspection gives the equalities
\begin{equation}
\label{eq:direct-sum-properties}
  (\sigma\oplus\tau)\oplus\upsilon
  =\sigma\oplus(\tau\oplus\upsilon),
  \qquad
  c(\sigma\oplus\tau)=c(\sigma)+c(\tau).
\end{equation}

\paragraph{The decoration product.}

The two components of $\mathcal D_m$ account for the two factors in
\eqref{eq:hm-tq}: ordinary permutations record the $q$-weighted cycle
statistic, while rooted permutations record the second cycle count.

For $s\geq0$, let
\[
  \mathfrak R_s:=\mathfrak S_{\{\ast,1,\ldots,s\}}
\]
be the set of permutations of $s+1$ letters with distinguished letter
$\ast$.  We call its elements \emph{rooted permutations}.  The cycle
containing $\ast$ has a unique cyclic notation
$(\ast,a_1,\ldots,a_u)$ beginning at the root.
In particular, $\mathfrak R_0$ consists of the single rooted permutation
$(\ast)$.

\begin{definition}[Product of rooted permutations]
\label[definition]{def:root-gluing-product}
Let $\tau\in\mathfrak R_a$ and $\upsilon\in\mathfrak R_b$.  Write their root
cycles on the original alphabets as
\[
  (\ast,a_1,\ldots,a_u),
  \qquad
  (\ast,b_1,\ldots,b_v),
\]
and then shift every nonroot letter of $\upsilon$ by $a$.  Define
$\tau\star\upsilon\in\mathfrak R_{a+b}$ by replacing the two root cycles with
\[
  (\ast,a_1,\ldots,a_u,a+b_1,\ldots,a+b_v)
\]
and retaining every nonroot cycle of $\tau$, together with every nonroot cycle
of $\upsilon$ after the same shift by $a$.
\end{definition}

The sequences following the roots concatenate, and the nonroot cycles form a
disjoint union.
Therefore $\star$ is associative, the identity permutation on $\{\ast\}$ is
its unit, and
\begin{equation*}
  c(\tau\star\upsilon)=c(\tau)+c(\upsilon)-1.
\end{equation*}
The last identity holds because the two root cycles become one and every
other cycle remains unchanged.

For ordinary permutations we use the direct sum $\sigma\oplus\sigma'$ defined
above.  Retaining $m^+$ and $m^-$ from
\cref{sec:characters}, put
\[
  \mathcal D_m:=\mathfrak S_{m^+}\times\mathfrak R_{m^-}
  \qquad(m\geq1).
\]
Give a decoration $\xi=(\sigma,\tau)\in\mathcal D_m$ the bidegree
\begin{equation}
\label{eq:decoration-bidegree}
  \deg_q(\xi):=m^+-c(\sigma),
  \qquad
  \deg_t(\xi):=c(\sigma)+c(\tau)-1.
\end{equation}
Thus $\deg_q(\xi)$ is the reflection length of the ordinary component, while
$\deg_t(\xi)$ counts the ordinary cycles together with the nonroot cycles
of the rooted component.

If $a$ and $b$ are not both odd, then
$(a+b)^+=a^++b^+$ and $(a+b)^-=a^-+b^-$.  In this case define
\begin{equation}
\label{eq:strict-decoration-product}
  (\sigma,\tau)\diamond(\sigma',\tau')
  :=(\sigma\oplus\sigma',\tau\star\tau')
  \in\mathcal D_{a+b}.
\end{equation}

Let $A_m^{\mathrm{dec}}=\kk[\mathcal D_m]$ for $m\geq1$, let
$A_0^{\mathrm{dec}}=\kk 1$, with $1$ a two-sided unit of bidegree $(0,0)$, and extend
\eqref{eq:strict-decoration-product}
bilinearly by
\begin{equation}
 A_a^{\mathrm{dec}}\diamond A_b^{\mathrm{dec}}=0
 \qquad\text{when $a$ and $b$ are both odd}.
 \label{eq:decoration-odd-square-zero}
\end{equation}

The product $\diamond$ makes
$A^{\mathrm{dec}}=\bigoplus_{m\geq0}A_m^{\mathrm{dec}}$ a connected bigraded
associative algebra.  This algebra is generally noncommutative.  On every
nonzero product, both degrees in
\eqref{eq:decoration-bidegree} are additive:
\[
  \deg_q(\xi\diamond\xi')
  =\deg_q(\xi)+\deg_q(\xi'),
  \qquad
  \deg_t(\xi\diamond\xi')
  =\deg_t(\xi)+\deg_t(\xi').
\]
Indeed, on three factors containing at most one odd size, associativity
follows from \eqref{eq:direct-sum-properties} and the associativity of
$\star$.  If at least two sizes are odd, both parenthesizations vanish by
\eqref{eq:decoration-odd-square-zero}.  The cycle formulas and the equalities
$(a+b)^+=a^++b^+$ and $(a+b)^-=a^-+b^-$ for every nonzero product prove
degree additivity.

Let
\[
 A^{\mathrm{even}}=\bigoplus_{m\text{ even}}A_m^{\mathrm{dec}},
 \qquad
 A^{\mathrm{odd}}=\bigoplus_{m\text{ odd}}A_m^{\mathrm{dec}}.
\]
Then $A^{\mathrm{dec}}=A^{\mathrm{even}}\oplus A^{\mathrm{odd}}$, and
$A^{\mathrm{odd}}$ is a square-zero $A^{\mathrm{even}}$-bimodule.

\paragraph{The decorated complex.}

\begin{definition}[Ordered set partitions with permutation decorations]
\label[definition]{def:rooted-cycle-decorated-osp}
For an arbitrary composition $\beta=(\beta_1,\ldots,\beta_k)\models n$, let
\[
  \operatorname{ROSP}_{\beta}[n]
  :=\left\{(B_1,\ldots,B_k;\xi_1,\ldots,\xi_k):
       \begin{array}{l}
       (B_1,\ldots,B_k)\in\operatorname{OSP}_{\beta}[n],\\
       \xi_i\in\mathcal D_{\beta_i}
       \end{array}\right\}.
\]
Let $M^{\beta}_{t,q}$ be its linear span, bigraded by the sums of the local
$t$- and $q$-degrees, with $\mathfrak S_n$ acting on block labels and fixing
the abstract decorations.  Define linearly
\[
 d_i^{t,q}(\mathbf B;\xi_1,\ldots,\xi_k)
 :=\begin{cases}
 0,&\beta_i\text{ and }\beta_{i+1}\text{ are both odd},\\
 (d_i\mathbf B;
   \xi_1,\ldots,\xi_i\diamond\xi_{i+1},\ldots,\xi_k),&\text{otherwise}.
 \end{cases}
\]
Thus every nonzero face sends a basis element to a basis element.
\end{definition}

For every nonempty composition $\alpha\models n$, set
\[
  C_k^{t,q}(\alpha)
  :=\bigoplus_{\substack{\beta\succeq\alpha\\\ell(\beta)=k}}
       M^{\beta}_{t,q},
  \qquad
  \partial_k^{t,q}:=\sum_{i=1}^{k-1}(-1)^{i-1}d_i^{t,q}.
\]
Then $(\partial^{t,q})^2=0$, and the differential has bidegree $(0,0)$.
Indeed, associativity of union and $\diamond$ gives the face identities, so the
standard alternating cancellation gives $(\partial^{t,q})^2=0$.  The cycle
formulas show that the differential has bidegree $(0,0)$.  Thus
$C_{\bullet}^{t,q}(\alpha)$ is a bigraded complex of
$\mathfrak S_n$-modules.

The local bigraded enumerator is
\begin{equation*}
  Q_m(t,q)
  :=\sum_{(\sigma,\tau)\in\mathcal D_m}
       q^{m^+-c(\sigma)}t^{c(\sigma)+c(\tau)-1}.
\end{equation*}

For every $m\geq1$,
\begin{equation}
\label{eq:Q-factorization}
  Q_m(t,q)
  =\prod_{i=0}^{m^+-1}(t+iq)
     \prod_{j=1}^{m^-}(t+j).
\end{equation}
In particular,
\begin{equation}
\label{eq:Q-minus-one}
  Q_m(t,-1)
  =t^{\underline{m^+}}\prod_{j=1}^{m^-}(t+j)
  =(t+m^-)^{\underline m}
  =m!\binom{t+m^-}{m}.
\end{equation}
Indeed, the ordinary factor is the classical cycle enumerator
\eqref{eq:cycle-expansion}.  For a rooted permutation on $s+1$ letters, the
cycle containing the root always contributes a factor $t$.  Removing this
factor from the usual cycle enumerator leaves
$\prod_{j=1}^{s}(t+j)$.  This proves \eqref{eq:Q-factorization}.  At $q=-1$,
the factors combine to give \eqref{eq:Q-minus-one}.

Choosing the ordered set partition by \eqref{eq:ordered-partition-module}
and then its block decorations shows that the bigraded Hilbert polynomial of
the $k$-th chain group is
\begin{equation}
\label{eq:Phi-k-tq}
  \Phi_{\alpha,k}(t,q)
  =\sum_{\substack{\beta\succeq\alpha\\\ell(\beta)=k}}
     \binom{n}{\beta_1,\ldots,\beta_k}
     \prod_{j=1}^{k}Q_{\beta_j}(t,q).
\end{equation}

\subsection{Decomposition and equivariant homology}

A composition has at most one odd part if and only if all adjacent products
in all of its coarsenings are nonzero.  Indeed, if there are at least two odd
parts, choose two consecutive odd parts; every intervening part is even, and
absorbing those even parts into either side produces adjacent odd parts in a
coarsening.  Hence every iterated decoration product attached to a coarsening
of such a composition is nonzero.  In particular, the condition holds for
every zigzag composition $\delta_n$.

Fix $\alpha\models n$ with at most one odd part and let
$\theta=(\sigma,\tau)\in\mathcal D_n$.  For $c\in D(\alpha)$, call $c$ a
\emph{factorization cut} of $\theta$ if
\begin{enumerate}[label=\textup{(\alph*)},leftmargin=*]
\item $\sigma([c^+])=[c^+]$; and
\item every nonroot cycle of $\tau$ lies entirely in $[c^-]$ or entirely in
      $\{c^-+1,\ldots,n^-\}$, while the root cycle
      $(\ast,a_1,\ldots,a_u)$ has all letters at most $c^-$ before all
      letters greater than $c^-$.
\end{enumerate}
The two sides of the cut cannot both have odd size, so these conditions are
equivalent to the existence of unique decorations
$\theta_L\in\mathcal D_c$ and
$\theta_R\in\mathcal D_{n-c}$ such that
$\theta=\theta_L\diamond\theta_R$.  Set
\begin{equation*}
 F_\alpha(\theta)
 :=\{c\in D(\alpha):c\text{ is a factorization cut of }\theta\},
 \qquad
 \gamma_\alpha(\theta):=\gamma_{F_\alpha(\theta)}.
\end{equation*}
For the ordinary component, condition~(a) is direct-sum factorization at
$c^+$, also called shifted concatenation.  Connected permutations and their
factorization under this product are discussed in
\cite[Sections~3.1--3.2, especially Equation~(34) and
Proposition~3.3]{HivertNovelliThibon2008}.  The lemma below combines this
ordinary factorization with rooted factorization at the same set of cuts.

\begin{lemma}[Simultaneous factorization of permutation decorations]
\label[lemma]{lem:decoration-cut-factorization}
Let $\beta\succeq\alpha$.  Iterated multiplication defines an injection
\[
 \mu_\beta:
 \mathcal D_{\beta_1}\times\cdots\times\mathcal D_{\beta_k}
 \longrightarrow\mathcal D_n,
 \qquad
 (\theta_1,\ldots,\theta_k)\longmapsto
 \theta_1\diamond\cdots\diamond\theta_k,
\]
whose image is precisely
$\{\theta\in\mathcal D_n:D(\beta)\subseteq F_\alpha(\theta)\}$.
\end{lemma}

\begin{proof}
Put $c_0=0$ and $c_j=\beta_1+\cdots+\beta_j$ for $1\leq j\leq k$.
Because $\beta$ has at most one odd part, $c_{j-1}$ and $\beta_j$ cannot both
be odd.  Applying the ceiling and floor identities to
$c_j=c_{j-1}+\beta_j$ therefore gives
\[
 c_j^+-c_{j-1}^+=\beta_j^+,
 \qquad
 c_j^--c_{j-1}^-=\beta_j^-.
\]
Let
\[
 I_j=\{c_{j-1}^++1,\ldots,c_j^+\},
 \qquad
 J_j=\{c_{j-1}^-+1,\ldots,c_j^-\}.
\]

Suppose first that every $c_1,\ldots,c_{k-1}$ is a factorization cut of
$\theta=(\sigma,\tau)$.  The nested prefix conditions
$\sigma([c_j^+])=[c_j^+]$ imply that each difference interval $I_j$ is
$\sigma$-invariant.  Restricting $\sigma$ to $I_j$ and translating this
interval to $[\beta_j^+]$ uniquely recovers a permutation
$\sigma_j\in\mathfrak S_{\beta_j^+}$.

For the rooted component, every nonroot cycle of $\tau$ is confined by all
the boundaries $c_j^-$ to a unique interval $J_j$.  The conditions on the root
cycle at the nested prefix cuts are jointly equivalent to the following
order condition: whenever $i<j$, every letter of $J_i$ in the root cycle
precedes every letter of $J_j$ in that cycle.  Taking from the root cycle the
subword formed by letters in $J_j$, together with the nonroot cycles contained
in $J_j$, and translating $J_j$ to $[\beta_j^-]$ therefore gives a unique
$\tau_j\in\mathfrak R_{\beta_j^-}$.  If the subword is empty, its root
cycle is the one-cycle $(\ast)$.  By construction,
\[
 (\sigma,\tau)
 = (\sigma_1,\tau_1)\diamond\cdots\diamond
   (\sigma_k,\tau_k).
\]

Conversely, an iterated product has invariant ordinary intervals, confined
nonroot cycles, and subwords from successive intervals in the required order,
so every $c_j$ is a factorization cut.  The restrictions above also show
that the local factors are unique.  This proves both injectivity and the
asserted image characterization simultaneously for all selected cuts.
\end{proof}

\begin{example}[A small zigzag example]
\label[example]{ex:delta-four-fibers}
Let $\alpha=\delta_4=(2,2)$.  Its only cut is $2$, while the ordinary and
rooted alphabets both have size $2$.  Consider the total decorations
\[
 \theta_1=(\id_2,(\ast)(1)(2)),\qquad
 \theta_2=(\id_2,(\ast,2,1)),\qquad
 \theta_3=((12),(\ast)(1)(2)).
\]
For $\theta_1$, the ordinary prefix $[1]$ is invariant, the two nonroot cycles
lie on opposite sides of the cut, and the condition on the root cycle is
vacuous.  Hence $F_{\delta_4}(\theta_1)=\{2\}$, and its fiber is
$C_\bullet((2,2))$ in bidegree $(4,0)$.  For $\theta_2$, the ordinary
condition still holds, but the letters in the root cycle occur in the order
$2,1$; thus $F_{\delta_4}(\theta_2)=\varnothing$, and its fiber is
$C_\bullet((4))$ in bidegree $(2,0)$.  For $\theta_3$, the transposition does
not preserve $[1]$, so the fiber is again $C_\bullet((4))$, now in bidegree
$(3,1)$.  The defects defined in \eqref{eq:factorization-defect} are
respectively $0$, $1$, and $0$.
\end{example}

\begin{theorem}[Factorization and homology of the decorated complex]
\label{thm:total-decoration-ribbon-homology}
Let $\alpha\models n$ be nonempty, with at most one odd part.  There is a
canonical bidegree-preserving, $\mathfrak S_n$-equivariant chain isomorphism
\begin{equation*}
 C_\bullet^{t,q}(\alpha)
 \cong
 \bigoplus_{\theta\in\mathcal D_n}
 C_\bullet(\gamma_\alpha(\theta)),
\end{equation*}
where the summand indexed by $\theta$ is placed in bidegree
$(\deg_t\theta,\deg_q\theta)$.  Moreover,
\begin{equation}
 \operatorname{ch}_{t,q}H_k(C_\bullet^{t,q}(\alpha))
 =\sum_{\substack{\theta\in\mathcal D_n\\
                   |F_\alpha(\theta)|=k-1}}
   t^{\deg_t\theta}q^{\deg_q\theta}
   r_{\gamma_\alpha(\theta)}.
 \label{eq:total-decoration-homology-character}
\end{equation}
Thus every bigraded homology Frobenius characteristic admits a nonnegative
expansion in ribbon Schur functions and is therefore Schur-positive.  The
unique nonzero homology module contributed by $\theta$ occurs in degree
$k=|F_\alpha(\theta)|+1$ and has dimension
\begin{equation*}
 \#\{\pi\in\mathfrak S_n:\Des(\pi)=F_\alpha(\theta)\}.
\end{equation*}
\end{theorem}

\begin{proof}
By \cref{lem:decoration-cut-factorization}, the permutation decorations
satisfy the factorization condition in
\cref{def:cut-factorizing-system}.  Apply
\cref{lem:factorization-fiber-decomposition}.  The dimension statement is
\eqref{eq:classical-descent-homology} for each summand.
\end{proof}

Continue to assume that $\alpha$ has at most one odd part.  For
$S\subseteq D(\alpha)$, set
\begin{equation}
 G_{\alpha,S}(t,q):=
 \sum_{S\subseteq T\subseteq D(\alpha)}
 (-1)^{|T|-|S|}\prod_iQ_{(\gamma_T)_i}(t,q).
 \label{eq:exact-cut-polynomial}
\end{equation}

\begin{corollary}[Decomposition by exact cut sets]
\label[corollary]{cor:exact-cut-decomposition}
Suppose that $\alpha\models n$ is nonempty and has at most one odd part.
The polynomial in \eqref{eq:exact-cut-polynomial} belongs to
$\mathbb N[t,q]$, and
\begin{align}
 G_{\alpha,S}(t,q)
 &=\sum_{\substack{\theta\in\mathcal D_n\\F_\alpha(\theta)=S}}
   t^{\deg_t\theta}q^{\deg_q\theta},
 \label{eq:factorization-cut-positive-count}\\
 \operatorname{ch}_{t,q}H_k(C_\bullet^{t,q}(\alpha))
 &=\sum_{\substack{S\subseteq D(\alpha)\\|S|=k-1}}
   G_{\alpha,S}(t,q)r_{\gamma_S}.
 \label{eq:explicit-ribbon-homology}
\end{align}
In top degree, writing $r=\ell(\alpha)$,
\begin{equation}
 \operatorname{ch}_{t,q}H_r(C_\bullet^{t,q}(\alpha))
 =\left(\prod_{i=1}^rQ_{\alpha_i}(t,q)\right)r_\alpha.
 \label{eq:top-homology}
\end{equation}
For $\alpha=\delta_n$ with $n\geq1$, this specializes to
\begin{equation*}
 \operatorname{ch}_{t,q}H_{\ell(\delta_n)}
 (C_\bullet^{t,q}(\delta_n))
 =t^{\lceil n/2\rceil}(1+t)^{\lfloor n/2\rfloor}r_{\delta_n},
\end{equation*}
which lies in $q$-degree zero.
\end{corollary}

\begin{proof}
Unique factorization and degree additivity show that the weight enumerator of
the decorations satisfying $T\subseteq F_\alpha(\theta)$ is
$\prod_iQ_{(\gamma_T)_i}(t,q)$.  Boolean M\"obius inversion gives
\eqref{eq:factorization-cut-positive-count}; grouping
\eqref{eq:total-decoration-homology-character} by exact cut set gives
\eqref{eq:explicit-ribbon-homology}.  At $S=D(\alpha)$, unique factorization
along all cuts gives the weight enumerator $\prod_iQ_{\alpha_i}$ and hence
\eqref{eq:top-homology}.  For $\delta_n$, use $Q_2=t(t+1)$ and $Q_1=t$.
\end{proof}

\subsection{The cut coproduct and total-decoration fibers}
\label{subsec:cut-coproduct}

The adjacent-merge differential admits a deconcatenation coproduct.  A cut at
a part boundary of $\alpha$ produces the corresponding prefix and suffix
compositions, so the natural object is the family over all compositions and
all finite label sets.

For a finite set $I$ with $|I|=n$, let $\mathfrak S_I$ denote its symmetric
group, and let $C_\bullet^{t,q}(\alpha)[I]$ denote the complex of
\cref{def:rooted-cycle-decorated-osp}, with $I$ in place of $[n]$.
For the empty composition, set
$C_0^{t,q}(\varnothing)[\varnothing]=\kk$ and set all other terms equal to
zero.  Relabeling the blocks along a bijection of finite sets makes
\begin{equation*}
 \mathcal C_\bullet[I]
 :=\begin{cases}
   \displaystyle\bigoplus_{\alpha\models |I|}
       C_\bullet^{t,q}(\alpha)[I],&I\neq\varnothing,\\[2mm]
   C_\bullet^{t,q}(\varnothing)[\varnothing],&I=\varnothing,
  \end{cases}
\end{equation*}
a species in bigraded chain complexes.  For species $\mathcal F$ and
$\mathcal G$, their Cauchy product is
\[
 (\mathcal F\mathbin{\cdot}\mathcal G)[I]
 =\bigoplus_{I=S\sqcup T}\mathcal F[S]\otimes\mathcal G[T],
\]
where the sum is over ordered decompositions; see
\cite{AguiarMahajan2010}.

Let $\alpha=(\alpha_1,\ldots,\alpha_r)$ and put
\[
 c_p=\alpha_1+\cdots+\alpha_p,\qquad
 \alpha_{\leq p}=(\alpha_1,\ldots,\alpha_p),\qquad
 \alpha_{>p}=(\alpha_{p+1},\ldots,\alpha_r),
\]
for $0\leq p\leq r$, with $c_0=0$ and with empty endpoint compositions.  Let
$I=S\sqcup T$ with $|S|=c_p$.  On a basis element
\[
 w=(B_1,\ldots,B_k;\xi_1,\ldots,\xi_k)
 \in C_k^{t,q}(\alpha)[I],
\]
define
\begin{equation}
 \Delta^p_{S,T}(w)=
 \begin{cases}
 (B_1,\ldots,B_j;\xi_1,\ldots,\xi_j)
 \otimes
 (B_{j+1},\ldots,B_k;\xi_{j+1},\ldots,\xi_k),
 \\
 \hfill\text{if }S=B_1\sqcup\cdots\sqcup B_j
       \text{ for some }0\leq j\leq k,\\[2mm]
 0,\hfill\text{otherwise}.
 \end{cases}
 \label{eq:cut-coproduct}
\end{equation}
The empty prefix and suffix are interpreted as the empty bar.  In the first
case, the two bars lie in
$C_\bullet^{t,q}(\alpha_{\leq p})[S]$ and
$C_\bullet^{t,q}(\alpha_{>p})[T]$, respectively.

For every ordered decomposition $I=S\sqcup T$, define the component
$\Delta_{S,T}:\mathcal C_\bullet[I]\to
\mathcal C_\bullet[S]\otimes\mathcal C_\bullet[T]$ on the summand indexed by
$\alpha$ by
\begin{equation*}
 \Delta_{S,T}\big|_{C_\bullet^{t,q}(\alpha)[I]}
 =\begin{cases}
   \Delta^p_{S,T},&|S|=c_p\text{ for some }0\leq p\leq r,\\
   0,&\text{otherwise}.
  \end{cases}
\end{equation*}
The index $p$, when it exists, is unique.  Let
$\Delta_I=\bigoplus_{I=S\sqcup T}\Delta_{S,T}$.  Let $\mathbf 1$
be the unit species, equal to $\kk$ on the empty set and zero otherwise, and
let $\epsilon:\mathcal C_\bullet\to\mathbf 1$ be the identity on
$\mathcal C_\bullet[\varnothing]=\kk$ and zero on nonempty sets.

\begin{proposition}[Species-level cut coproduct]
\label{prop:cut-coproduct}
Each map $\Delta^p_{S,T}$ is natural under relabeling, preserves the two weight
degrees, and is a degree-zero chain map
\begin{equation}
 \Delta^p_{S,T}:C_k^{t,q}(\alpha)[I]
 \longrightarrow
 \bigoplus_{a+b=k}
 C_a^{t,q}(\alpha_{\leq p})[S]\otimes
 C_b^{t,q}(\alpha_{>p})[T],
 \label{eq:cut-coproduct-chain-map}
\end{equation}
where the target has the usual tensor-product differential.  The natural
transformation
\[
 \Delta:\mathcal C_\bullet\longrightarrow
 \mathcal C_\bullet\mathbin{\cdot}\mathcal C_\bullet
\]
is coassociative and counital.  Explicitly, for every ordered decomposition
$I=R\sqcup S\sqcup T$,
\begin{equation}
 (\Delta_{R,S}\otimes\id)\Delta_{R\sqcup S,T}
 =(\id\otimes\Delta_{S,T})\Delta_{R,S\sqcup T}.
 \label{eq:cut-coproduct-coassociativity}
\end{equation}
The endpoint components satisfy
\begin{equation}
 (\epsilon\otimes\id)\Delta_{\varnothing,I}=\id_{\mathcal C_\bullet[I]},
 \qquad
 (\id\otimes\epsilon)\Delta_{I,\varnothing}=\id_{\mathcal C_\bullet[I]}.
 \label{eq:cut-coproduct-counit}
\end{equation}
Thus $(\mathcal C_\bullet,\Delta,\epsilon)$ is a counital differential graded
comonoid in the Cauchy monoidal category of species.
\end{proposition}

\begin{proof}
Suppose first that $S$ is not a union of initial blocks of $w$.  Merging
adjacent blocks deletes a block boundary and cannot create such an initial
union, so both sides of the chain-map identity vanish on $w$.

Now suppose that $S=B_1\sqcup\cdots\sqcup B_j$.  A face with index $i<j$
acts in the left tensor factor.  A face with index $i>j$ acts in the right
tensor factor, and its tensor-product sign is
\[
 (-1)^j(-1)^{i-j-1}=(-1)^{i-1},
\]
the sign of the same face in the source.  The face with index $i=j$ merges a
block in $S$ with a block in $T$.  Its image under $\Delta^p_{S,T}$ is zero,
and there is no corresponding term in the tensor-product differential.  This
proves \eqref{eq:cut-coproduct-chain-map}; naturality and bidegree preservation
are immediate.

For $I=R\sqcup S\sqcup T$, restrict both sides of
\eqref{eq:cut-coproduct-coassociativity} to a composition summand.  They vanish
unless $R$ and $R\sqcup S$ have sizes equal to two part boundaries of that
composition and are unions of the corresponding initial blocks.  When these
conditions hold, both sides return the same three consecutive bars.  This
proves coassociativity.  The two endpoint cuts give
\eqref{eq:cut-coproduct-counit}.
\end{proof}

For a nonempty composition $\alpha=(\alpha_1,\ldots,\alpha_r)$, define its
reduced cut coproduct by omitting the endpoint components:
\begin{equation}
\begin{split}
 \overline\Delta_{\alpha,I}:C_\bullet^{t,q}(\alpha)[I]
 &\longrightarrow
 \bigoplus_{p=1}^{r-1}
 \ \bigoplus_{\substack{I=S\sqcup T\\|S|=c_p}}
 C_\bullet^{t,q}(\alpha_{\leq p})[S]\otimes
 C_\bullet^{t,q}(\alpha_{>p})[T],\\
 \overline\Delta_{\alpha,I}
 &:=\bigoplus_{p=1}^{r-1}
   \ \bigoplus_{\substack{I=S\sqcup T\\|S|=c_p}}
   \Delta^p_{S,T}.
\end{split}
 \label{eq:reduced-cut-coproduct}
\end{equation}
When $r=1$, the target is zero and so is the map.

Assume from now on that $\alpha$ has at most one odd part.  For
$\theta\in\mathcal D_n$, let
$C_\bullet^{t,q}(\alpha;\theta)[I]$ denote the subcomplex spanned by basis
elements with total decoration $\theta$.

\begin{theorem}[Coproduct compatibility of total decorations]
\label{thm:total-decoration-cut-coproduct}
Let $\alpha\models n$ have at most one odd part, set $r=\ell(\alpha)$, let
$1\leq p<r$, and put $c=c_p$.  For every $I=S\sqcup T$ with $|S|=c$ and every
$\theta\in\mathcal D_n$, the following hold.
\begin{enumerate}[label=\textup{(\roman*)},leftmargin=*]
\item If $c\notin F_\alpha(\theta)$, then
\[
 \Delta^p_{S,T}\bigl(C_\bullet^{t,q}(\alpha;\theta)[I]\bigr)=0.
\]
\item If $c\in F_\alpha(\theta)$, write the unique factorization as
$\theta=\theta_L\diamond\theta_R$, with
$\theta_L\in\mathcal D_c$ and $\theta_R\in\mathcal D_{n-c}$.  Then
\begin{equation}
 \Delta^p_{S,T}\bigl(C_\bullet^{t,q}(\alpha;\theta)[I]\bigr)
 \subseteq
 C_\bullet^{t,q}(\alpha_{\leq p};\theta_L)[S]\otimes
 C_\bullet^{t,q}(\alpha_{>p};\theta_R)[T].
 \label{eq:total-decoration-coproduct-inclusion}
\end{equation}
Moreover,
\begin{align}
 F_\alpha(\theta)
 &=F_{\alpha_{\leq p}}(\theta_L)
   \sqcup\{c\}
   \sqcup\bigl(c+F_{\alpha_{>p}}(\theta_R)\bigr),
 \label{eq:factorization-cut-locality}\\
 \gamma_\alpha(\theta)
 &=\gamma_{\alpha_{\leq p}}(\theta_L)
   \mathbin{\cdot}
   \gamma_{\alpha_{>p}}(\theta_R).
 \label{eq:ribbon-concatenation-at-cut}
\end{align}
Here $c+U=\{c+u:u\in U\}$ and $\cdot$ denotes concatenation of compositions.
\end{enumerate}
For each $c_p\in F_\alpha(\theta)$, denote the corresponding factors by
$\theta_L^{(p)}$ and $\theta_R^{(p)}$.  Then
\begin{equation}
 \overline\Delta_{\alpha,I}
 \bigl(C_\bullet^{t,q}(\alpha;\theta)[I]\bigr)
 \subseteq
 \bigoplus_{\substack{1\leq p<\ell(\alpha)\\c_p\in F_\alpha(\theta)}}
 \ \bigoplus_{\substack{I=S\sqcup T\\|S|=c_p}}
 C_\bullet^{t,q}(\alpha_{\leq p};\theta_L^{(p)})[S]
 \otimes
 C_\bullet^{t,q}(\alpha_{>p};\theta_R^{(p)})[T].
 \label{eq:total-decoration-coproduct-decomposition}
\end{equation}
Thus the reduced coproduct respects the total-decoration decomposition, with
the decoration index split by unique factorization.
\end{theorem}

\begin{proof}
A basis element in the $\theta$-fiber has block type $\beta$ satisfying
$D(\beta)\subseteq F_\alpha(\theta)$ by
\cref{lem:decoration-cut-factorization}.  If its image under
$\Delta^p_{S,T}$ is nonzero, then $c\in D(\beta)$, which proves~\textup{(i)}.

Suppose $c\in F_\alpha(\theta)$.  If a bar splits at $c$, the products of its
left and right local decorations factor $\theta$ at $c$.  Uniqueness in
\cref{lem:decoration-cut-factorization} identifies them with $\theta_L$ and
$\theta_R$, proving \eqref{eq:total-decoration-coproduct-inclusion}.

A cut $d<c$ factors $\theta$ if and only if simultaneous factorization at
$d$ and $c$ exists.  By associativity and uniqueness, this is equivalent to
factorization of $\theta_L$ at $d$.  Similarly, a cut $d>c$ factors $\theta$
if and only if $\theta_R$ factors at $d-c$.  This proves
\eqref{eq:factorization-cut-locality}; taking cut sets gives
\eqref{eq:ribbon-concatenation-at-cut}.  Summing the component inclusions over
all internal cuts gives
\eqref{eq:total-decoration-coproduct-decomposition}.
\end{proof}

At the level of homology, the aggregate label-set component at a
factorization cut is injective.  Its cokernel is determined by the ribbon
product rule.

For nonempty compositions $\lambda=(\lambda_1,\ldots,\lambda_a)$ and
$\mu=(\mu_1,\ldots,\mu_b)$, write
\[
 \lambda\odot\mu
 = (\lambda_1,\ldots,\lambda_{a-1},
    \lambda_a+\mu_1,\mu_2,\ldots,\mu_b)
\]
for their near-concatenation.

\begin{corollary}[Cut-induced ribbon injection]
\label{cor:ribbon-cut-exact-sequence}
Retain the hypotheses and notation of
\cref{thm:total-decoration-cut-coproduct}, assume $c\in F_\alpha(\theta)$,
and set
\[
 \lambda=\gamma_{\alpha_{\leq p}}(\theta_L),\qquad
 \mu=\gamma_{\alpha_{>p}}(\theta_R).
\]
Let
\[
 \Delta_c=\bigoplus_{\substack{S\subseteq I\\|S|=c}}
 \Delta^p_{S,I\setminus S}.
\]
Under the K\"unneth identification, the induced homology map is an injective
$\mathfrak S_I$-map
\begin{equation}
 H(\Delta_c):
 H_{\ell(\lambda)+\ell(\mu)}
   (C_\bullet^{t,q}(\alpha;\theta)[I])
 \lhook\joinrel\longrightarrow
 \bigoplus_{\substack{S\subseteq I\\|S|=c}}
 H_{\ell(\lambda)}
   (C_\bullet^{t,q}(\alpha_{\leq p};\theta_L)[S])
 \otimes
 H_{\ell(\mu)}
   (C_\bullet^{t,q}(\alpha_{>p};\theta_R)[I\setminus S]).
 \label{eq:ribbon-cut-homology-injection}
\end{equation}
All terms lie in bidegree
$\operatorname{bideg}(\theta)=
 \operatorname{bideg}(\theta_L)+\operatorname{bideg}(\theta_R)$.
If $Q_{\theta,c}$ denotes the cokernel, then
\begin{equation}
 \operatorname{ch}_{t,q}Q_{\theta,c}
 =t^{\deg_t\theta}q^{\deg_q\theta}r_{\lambda\odot\mu}.
 \label{eq:ribbon-cut-cokernel-character}
\end{equation}
Thus $Q_{\theta,c}$ is noncanonically isomorphic to the classical ribbon
representation indexed by $\lambda\odot\mu$, placed in the bidegree of
$\theta$.
\end{corollary}

\begin{proof}
The complexes are finite-dimensional over a field, so the K\"unneth theorem
identifies the homology of each tensor product with the tensor product of the
factor homologies.  Both factors have homology concentrated in their top
degrees by \cref{prop:classical-rank-selected-homology}.

By \eqref{eq:factorization-cut-locality}, the source fiber is the classical
ribbon complex for $\lambda\mathbin{\cdot}\mu$.  Its top chain group has block
type $\lambda\mathbin{\cdot}\mu$.  Every top basis element has a unique prefix
union of size $c$, and cutting there gives a pair of top basis elements of
types $\lambda$ and $\mu$.  Conversely, concatenating such a pair recovers the
source basis element, and unique factorization recovers the local decorations.
Therefore $\Delta_c$ is a bijection on top chain groups.

The target chain complex has no group above total degree
$\ell(\lambda)+\ell(\mu)$.  Hence a top cycle whose image is zero in homology
already has zero image as a chain.  The top-chain bijection proves
injectivity.  The target is the induced product of the two factor modules, so
its Frobenius characteristic is $r_\lambda r_\mu$, while the source has
character $r_{\lambda\cdot\mu}$.  The ribbon product rule
\[
 r_\lambda r_\mu
 =r_{\lambda\cdot\mu}+r_{\lambda\odot\mu}
\]
gives \eqref{eq:ribbon-cut-cokernel-character}.  In characteristic zero, the
Frobenius characteristic determines the isomorphism class of a finite
$\mathfrak S_I$-module.
\end{proof}

\begin{corollary}[Cut-primitive homology]
\label{cor:cut-primitive-homology}
Let $\alpha\models n$ have at most one odd part.  Define
\begin{equation}
 \operatorname{Prim}_{\mathrm{cut}}H(C_\bullet^{t,q}(\alpha))
 :=\ker H(\overline\Delta_{\alpha,[n]}).
 \label{eq:cut-primitive-definition}
\end{equation}
If $\ell(\alpha)=1$, the reduced coproduct is zero by convention.  Then
\begin{equation}
 \operatorname{Prim}_{\mathrm{cut}}H(C_\bullet^{t,q}(\alpha))
 =H_1(C_\bullet^{t,q}(\alpha)),
 \label{eq:cut-primitives-equal-h-one}
\end{equation}
and
\begin{equation}
 \operatorname{ch}_{t,q}\operatorname{Prim}_{\mathrm{cut}}
 H(C_\bullet^{t,q}(\alpha))
 =G_{\alpha,\varnothing}(t,q)r_{(n)}.
 \label{eq:cut-primitive-character}
\end{equation}
Equivalently, a total-decoration ribbon summand is cut-primitive exactly when
its factorization-cut set is empty.
\end{corollary}

\begin{proof}
If $F_\alpha(\theta)=\varnothing$, part~\textup{(i)} of
\cref{thm:total-decoration-cut-coproduct} makes every internal component of
the coproduct zero on the $\theta$-fiber.  If
$F_\alpha(\theta)\neq\varnothing$, choose $c\in F_\alpha(\theta)$.  The
aggregate component at $c$ is injective on the homology of that fiber by
\cref{cor:ribbon-cut-exact-sequence}.  The target decomposes by the unique
factor pair, so images from distinct total decorations cannot cancel.  Thus
the kernel of $H(\overline\Delta_{\alpha,[n]})$ is precisely the direct sum of
the fibers with empty factorization-cut set.  By
\cref{thm:total-decoration-ribbon-homology}, these are exactly the fibers in
homological degree one.  Grouping them by the empty exact cut set and applying
\cref{cor:exact-cut-decomposition} proves
\eqref{eq:cut-primitives-equal-h-one} and
\eqref{eq:cut-primitive-character}.
\end{proof}

\begin{remark}[Scope of the comonoid structure]
\label{rem:coalgebra-not-fixed-hopf}
The coproduct sends a composition to a prefix and a suffix, so an individual
summand $C_\bullet^{t,q}(\alpha)$ is not a subcomonoid.  The comonoid is the
species $\mathcal C_\bullet$ obtained by summing over all compositions.  The
algebra $A^{\mathrm{dec}}$ organizes multiplication of local decorations, but
no compatible product on the labelled complexes is defined here; accordingly,
no Hopf-monoid structure is asserted.
\end{remark}

\subsection{Bigraded Euler characteristic}

Since $Q_m(t,q)/m!=\mathsf h_m(t,q)$, combining
\eqref{eq:Phi-k-tq} with the ribbon coarsening formula gives, for every
nonempty composition $\alpha\models n$,
\begin{equation}
 \sum_{k=1}^{\ell(\alpha)}
 (-1)^{\ell(\alpha)-k}\Phi_{\alpha,k}(t,q)
 =n!\,Z_\alpha(t,q).
 \label{eq:bigraded-euler}
\end{equation}
Thus $n!Z_\alpha(t,q)$ is the shifted bigraded Euler characteristic of the
chain groups.  When $\alpha$ has at most one odd part, Euler--Poincar\'e and
\cref{thm:total-decoration-ribbon-homology} also express it as the alternating
sum of the explicitly determined homology groups.

Specializing $q=-1$ takes the signed difference between even and odd
$q$-degrees.  For zigzag compositions, the next section identifies
this specialization with $n!\Omega(P_n;t)$ and retains the reflection-length variable.

\section{Alternating fences and bivariate positivity}
\label{sec:staircase-consequences}

We first identify the one-variable specialization with the alternating-fence
order polynomial and then prove coefficientwise positivity for its bivariate
refinement.

\subsection{The alternating-fence specialization}

For \(n\geq1\), let \(P_n\) be the alternating fence on
\(\{x_1,\ldots,x_n\}\), with cover relations
\begin{equation*}
 x_1<x_2>x_3<x_4>\cdots.
\end{equation*}
For every positive integer $m$, its order polynomial $\Omega(P_n;m)$ counts
weakly order-preserving maps $f:P_n\to[m]$ satisfying
$x\leq_P y\Rightarrow f(x)\leq f(y)$.  Put $\Omega(P_0;t)=1$.

Set
\begin{equation*}
 b_a(t)=\binom{t+a}{2a}\quad(a\geq0),\qquad
 d_a(t)=\binom{t+a}{2a+1}\quad(a\geq0).
\end{equation*}
Thus \(b_0(t)=1\) and \(d_0(t)=t\).

For later use, set
\[
 E_N(t,q)=Z_{\delta_{2N}}(t,q),\qquad
 O_N(t,q)=Z_{\delta_{2N+1}}(t,q).
\]
The coarsening formula gives
\begin{equation}
 E_N(t,q)
 =\sum_{\alpha\models N}(-1)^{N-\ell(\alpha)}
  \prod_{j=1}^{\ell(\alpha)}\mathsf h_{2\alpha_j}(t,q),
 \label{eq:even-composition-expansion}
\end{equation}
and hence
\begin{equation}
 E_N(t,q)=\sum_{a=1}^{N}(-1)^{a+1}
 \mathsf h_{2a}(t,q)E_{N-a}(t,q),\qquad E_0(t,q)=1.
 \label{eq:E-recurrence-q}
\end{equation}
Similarly, grouping a coarsening of \((2^N,1)\) by the number \(a\) of
twos merged with its final one gives
\begin{equation}
 O_N(t,q)=\sum_{a=0}^{N}(-1)^a
 \mathsf h_{2a+1}(t,q)E_{N-a}(t,q).
 \label{eq:odd-even-convolution-q}
\end{equation}

The following known identification is an instance of the Kreweras
determinant.

\begin{proposition}[Alternating-fence specialization]
\label{prop:staircase-specialization}
For every \(n\geq0\),
\begin{equation}
 \Omega(P_n;t)=Z_{\delta_n}(t,-1)=\phi_t(R_{\delta_n}).
 \label{eq:staircase-specialization}
\end{equation}
\end{proposition}

\begin{proof}
The case $n=0$ is immediate.  Kreweras's determinant, in the form
\cite[Proposition~2.2]{FerroniMoralesPanova}, applies to the two zigzag
ribbons in \cite[Equation~(5.1)]{FerroniMoralesPanova}.  Their cell posets are
dual to $P_{2N}$ and $P_{2N+1}$, respectively, and duality preserves the weak
order polynomial.  Substitution in the determinant gives, with
$\Delta_0(t)=1$ and $b_a=0$ for $a<0$,
\begin{equation}
 \Omega(P_{2N};t)=\Delta_N(t),
 \qquad
 \Delta_N(t)=\det[b_{j-i+1}(t)]_{i,j=1}^{N}.
 \label{eq:fence-even-determinant}
\end{equation}
For the odd ribbon, the same substitution gives the Hessenberg determinant
\begin{equation*}
 \Omega(P_{2N+1};t)=
 \det\!\begin{pmatrix}
 b_1&b_2&\cdots&b_N&d_N\\
 1&b_1&\cdots&b_{N-1}&d_{N-1}\\
 0&1&\cdots&b_{N-2}&d_{N-2}\\
 \vdots&\ddots&\ddots&\vdots&\vdots\\
 0&\cdots&0&1&d_0
 \end{pmatrix}.
\end{equation*}
Expansion along the last column gives
\begin{equation}
 \Omega(P_{2N+1};t)
 =\sum_{a=0}^{N}(-1)^a d_a(t)\Delta_{N-a}(t).
 \label{eq:fence-odd-determinant}
\end{equation}
Expanding \eqref{eq:fence-even-determinant} along the first row yields
\[
 \Delta_N(t)=\sum_{a=1}^{N}(-1)^{a+1}b_a(t)\Delta_{N-a}(t).
\]
By \eqref{eq:even-kernel}, this is exactly
\eqref{eq:E-recurrence-q} at $q=-1$, so
$\Delta_N(t)=E_N(t,-1)$.  Substituting this equality and
\eqref{eq:odd-kernel} into \eqref{eq:fence-odd-determinant} gives
\eqref{eq:odd-even-convolution-q} at $q=-1$.  Hence
$\Omega(P_n;t)=Z_{\delta_n}(t,-1)$ in both parities.
\end{proof}
\subsection{Generating functions}

Introduce
\begin{equation*}
 \mathcal B_{t,q}(x)
 =\sum_{a\geq0}\mathsf h_{2a}(t,q)x^a,
 \qquad
 \mathcal O_{t,q}(x)
 =\sum_{a\geq0}\mathsf h_{2a+1}(t,q)x^a.
\end{equation*}
We use the rising Pochhammer symbol
\[
 (a)_n=a(a+1)\cdots(a+n-1),
 \qquad (a)_0=1,
\]
and
\[
 {}_2F_1(A,B;C;x)
 =\sum_{n\geq0}
  \frac{(A)_n(B)_n}
       {(C)_n n!}x^n.
\]

\begin{lemma}[Generating functions for the zigzag compositions]
\label{lem:staircase-generating-series}
In \(\mathcal R[[x]]\),
\begin{align*}
 \sum_{N\geq0}E_N(t,q)x^N
 &=\frac{1}{\mathcal B_{t,q}(-x)},
 \\
 \sum_{N\geq0}O_N(t,q)x^N
 &=\frac{\mathcal O_{t,q}(-x)}{\mathcal B_{t,q}(-x)}.
\end{align*}
Moreover, in \(\mathbb Q[t,q,q^{-1}][[x]]\),
\begin{equation}
 \mathcal B_{t,q}(-x)
 ={}_2F_1\left(\frac tq,t+1;\frac12;-\frac{qx}{4}\right).
 \label{eq:B-hypergeom}
\end{equation}
Every coefficient of these identities belongs to $\mathbb Q[t,q]$.
Hence they have a unique coefficientwise
specialization at \(q=0\), even though the displayed hypergeometric
parameters contain \(q^{-1}\).
\end{lemma}

\begin{proof}
Multiplying \eqref{eq:E-recurrence-q} by \(x^N\) and summing gives
\(\sum E_Nx^N=\mathcal B_{t,q}(-x)^{-1}\).
Equation~\eqref{eq:odd-even-convolution-q} then gives
\(\sum O_Nx^N=\mathcal O_{t,q}(-x)\sum E_Nx^N\).
The Pochhammer and duplication identities give, coefficientwise,
\begin{equation*}
 \mathsf h_{2a}(t,q)
 =\frac{(t/q)_a(t+1)_a}
        {(1/2)_a a!}\left(\frac q4\right)^a.
\end{equation*}
Summation proves \eqref{eq:B-hypergeom}.
Finally, \eqref{eq:hm-tq} places both kernel series in
\(\mathbb Q[t,q][[x]]\); the even kernel has constant term \(1\).
Its inverse lies there as well, which proves the assertion at \(q=0\).
\end{proof}

\subsection{Coefficientwise positivity}

Ribbon inclusion--exclusion introduces alternating signs.  After the
substitution $q=-u$, multiplying the zigzag evaluation by $n!$ gives a
polynomial with nonnegative coefficients.

\begin{theorem}[Bivariate coefficientwise positivity]
\label{thm:cycle-sign-coherence}
For every \(n\geq0\),
\begin{equation}
 n!\,Z_{\delta_n}(t,-u)\in\mathbb N[t,u].
 \label{eq:cycle-sign-coherence}
\end{equation}
Equivalently, whenever
\([t^iq^j]\,n!Z_{\delta_n}(t,q)\neq0\), its sign is \((-1)^j\).
\end{theorem}

At \(u=1\), \cref{thm:cycle-sign-coherence} specializes to coefficientwise
nonnegativity of \(n!\Omega(P_n;t)\), also implied by Kahane's
permutation-statistic theorem for arbitrary fence order polynomials
\cite{KahaneFenceCoefficients}.  The theorem above also records the
$q$-degree $m^+-c(\sigma)$, which is lost when $u=1$.  Its proof uses the
following recurrence for the even generating function.

\begin{lemma}[Positive Riccati recurrences]
\label{lem:bivariate-log-derivative-positivity}
Put
\[
 B(x)=\mathcal B_{t,-u}(-x),\qquad
 A(x)=-\frac{B'(x)}{B(x)}=\sum_{n\geq0}a_nx^n,
 \qquad
 e_n=4a_{n+1}-ua_n.
\]
Then
\[
 a_n,e_n\in\mathbb Q_{\geq0}[t,u]\qquad(n\geq0).
\]
\end{lemma}

\begin{proof}
By \cref{lem:staircase-generating-series},
\[
 B(x)={}_2F_1\left(-\frac tu,t+1;\frac12;\frac{ux}{4}\right).
\]
Gauss's differential equation
\cite[Eq.~(15.10.1)]{NISTHandbook2010}, after the substitution \(z=ux/4\),
becomes
\begin{equation}
 (4x-ux^2)B''
 +\bigl(2-(tu-t+2u)x\bigr)B'
 +t(t+1)B=0.
 \label{eq:even-kernel-ode}
\end{equation}
Although the displayed hypergeometric parameters contain \(u^{-1}\),
\eqref{eq:even-kernel-ode} is an identity in
\(\mathbb Q[t,u][[x]]\): it also follows directly from the ratio of
consecutive coefficients of \(B\).  In particular, it remains valid at
\(u=0\).

Substitution of \(B'=-AB\) into \eqref{eq:even-kernel-ode} gives
\begin{equation}
 (4x-ux^2)A'
 =(4x-ux^2)A^2
 -\bigl(2-(tu-t+2u)x\bigr)A+t(t+1).
 \label{eq:even-kernel-riccati}
\end{equation}
Set
\[
 \kappa_m=\sum_{i=0}^m a_i a_{m-i}\quad(m\geq0),
 \qquad \kappa_m=0\quad(m<0).
\]
The constant term in \eqref{eq:even-kernel-riccati} is
\[
 a_0=\frac{t(t+1)}2.
\]
For \(n\geq1\), its coefficient of \(x^n\) gives
\begin{equation}
 (4n+2)a_n
 =4\kappa_{n-1}-u\kappa_{n-2}
  +\bigl(tu-t+(n+1)u\bigr)a_{n-1}.
 \label{eq:raw-a-recurrence}
\end{equation}
Using the definition of $e_i$, we may rewrite these terms as
\[
 4\kappa_{n-1}-u\kappa_{n-2}
 =\sum_{i=0}^{n-2}a_i e_{n-2-i}+4a_0a_{n-1},
\]
so \eqref{eq:raw-a-recurrence} becomes
\begin{equation}
 (4n+2)a_n
 =\sum_{i=0}^{n-2}a_i e_{n-2-i}
  +\bigl(2t^2+t+tu+(n+1)u\bigr)a_{n-1}.
 \label{eq:positive-a-recurrence}
\end{equation}
Here and below an empty sum is zero.  Applying
\eqref{eq:raw-a-recurrence} with index \(n+1\), multiplying by four, and
subtracting \((4n+6)ua_n\) gives
\[
 (4n+6)e_n
 =4\bigl(4\kappa_n-u\kappa_{n-1}\bigr)
  +(4tu-4t+2u)a_n.
\]
Since
\[
 4\kappa_n-u\kappa_{n-1}
 =\sum_{i=0}^{n-1}a_i e_{n-1-i}+4a_0a_n,
\]
and \(4a_0=2t(t+1)\), we obtain
\begin{equation}
 (4n+6)e_n
 =4\sum_{i=0}^{n-1}a_i e_{n-1-i}
  +\bigl(8t^2+4t+4tu+2u\bigr)a_n
 \qquad(n\geq0).
 \label{eq:positive-e-recurrence}
\end{equation}
Now \(a_0\in\mathbb Q_{\geq0}[t,u]\);
\eqref{eq:positive-a-recurrence} at \(n=1\) gives \(a_1\geq0\), and
\eqref{eq:positive-e-recurrence} at \(n=0\) gives \(e_0\geq0\).
If \(a_0,\ldots,a_n\) and \(e_0,\ldots,e_{n-1}\) are coefficientwise
nonnegative for \(n\geq1\), then \eqref{eq:positive-a-recurrence} at
\(n+1\) gives \(a_{n+1}\geq0\), while
\eqref{eq:positive-e-recurrence} at \(n\) gives \(e_n\geq0\).
Simultaneous induction proves the lemma.
\end{proof}

\begin{proof}[Proof of \cref{thm:cycle-sign-coherence}]
Let
\[
 F_{\mathrm{even}}(x)=\sum_{N\geq0}Z_{\delta_{2N}}(t,-u)x^N,
 \qquad
 F_{\mathrm{odd}}(x)=\sum_{N\geq0}Z_{\delta_{2N+1}}(t,-u)x^N.
\]
By \cref{lem:staircase-generating-series}, \(F_{\mathrm{even}}=1/B\).
Hence \(F_{\mathrm{even}}'=AF_{\mathrm{even}}\) and
\(F_{\mathrm{even}}(0)=1\).  If
\(F_{\mathrm{even}}=\sum_{N\geq0}f_Nx^N\), then
\[
 (N+1)f_{N+1}=\sum_{i+j=N}a_i f_j.
\]
The lemma shows inductively that every \(f_N\) lies in
\(\mathbb Q_{\geq0}[t,u]\).

For the odd series, set \(J(x)=\mathcal O_{t,-u}(-x)\).  The local identity
\begin{equation*}
 (2m+1)\mathsf h_{2m+1}(t,-u)
 =(t-mu)\mathsf h_{2m}(t,-u)
\end{equation*}
gives, after summation with the kernel signs,
\[
 2xJ'(x)+J(x)=tB(x)-uxB'(x).
\]
Since \(F_{\mathrm{odd}}=J/B\) and \(B'/B=-A\), this is equivalent to
\begin{equation*}
 2xF_{\mathrm{odd}}'(x)+F_{\mathrm{odd}}(x)
 =2xA(x)F_{\mathrm{odd}}(x)+t+uxA(x).
\end{equation*}
Writing \(F_{\mathrm{odd}}=\sum_{N\geq0}o_Nx^N\), we obtain
\[
 o_0=t,\qquad
 (2N+1)o_N
 =2\sum_{i+j=N-1}a_i o_j+ua_{N-1}\quad(N\geq1).
\]
Another induction gives \(o_N\in\mathbb Q_{\geq0}[t,u]\).  We have therefore
proved
\[
 Z_{\delta_n}(t,-u)\in\mathbb Q_{\geq0}[t,u]\qquad(n\geq0).
\]

The case $n=0$ is immediate.  For $n\geq1$, put
\(Q_m(t,q)=m!\mathsf h_m(t,q)\in\mathbb Z[t,q]\).  The coarsening formula
gives, for every nonempty composition \(\alpha\models n\),
\[
 n!Z_\alpha(t,q)
 =\sum_{\beta\succeq\alpha}
 (-1)^{\ell(\alpha)-\ell(\beta)}
 \binom{n}{\beta_1,\ldots,\beta_{\ell(\beta)}}
 \prod_i Q_{\beta_i}(t,q)
 \in\mathbb Z[t,q].
\]
Applying this to $\alpha=\delta_n$ and combining it with the coefficientwise
nonnegativity over $\mathbb Q$ proved above gives
\eqref{eq:cycle-sign-coherence}.
\end{proof}

\subsection{Factorization defect and the support region}

We now connect the homology decomposition with the sign pattern in
\cref{thm:cycle-sign-coherence}.  Put
\[
 \ell_n=\ell(\delta_n)=\left\lceil\frac n2\right\rceil,
 \qquad
 H_k^{i,j}(n)=H_k(C_\bullet^{t,q}(\delta_n))_{(i,j)}.
\]
For a total decoration $\theta=(\sigma,\tau)\in\mathcal D_n$, define its
\emph{factorization defect} by
\begin{equation}
 \operatorname{fd}_n(\theta)
 :=c(\sigma)-|F_{\delta_n}(\theta)|-1.
 \label{eq:factorization-defect}
\end{equation}
For a basis element of $C_\bullet^{t,q}(\delta_n)$, multiply its local
decorations in block order and call the result its total decoration.  This
product is nonzero because every coarsening of $\delta_n$ has at most one odd
part.  For an
integer $d$, let $C_\bullet^{t,q}(\delta_n;d)$ be the span of the basis
elements whose total decoration has factorization defect $d$.  We will show
that these spaces are subcomplexes and vanish for $d<0$.  For $d\geq0$, set
\begin{equation}
 \mathcal H_{n,d}(t,u)
 :=\sum_{\substack{\theta\in\mathcal D_n\\
                    \operatorname{fd}_n(\theta)=d}}
 t^{\deg_t\theta}u^{\deg_q\theta}
 r_{\gamma_{\delta_n}(\theta)}.
 \label{eq:defect-layer}
\end{equation}
This is a homogeneous symmetric function of degree $n$, presented as a
nonnegative linear combination of ribbon Schur functions with coefficients in
$\mathbb N[t,u]$.

\begin{theorem}[Defect decomposition and support]
\label{thm:factorization-defect-layers}
Let $n\geq1$.  Every total decoration satisfies
$0\leq\operatorname{fd}_n(\theta)\leq\ell_n-1$.  The differential preserves
total decoration and gives a canonical decomposition into bigraded
$\mathfrak S_n$-subcomplexes
\begin{equation}
 C_\bullet^{t,q}(\delta_n)
 =\bigoplus_{d=0}^{\ell_n-1}C_\bullet^{t,q}(\delta_n;d).
 \label{eq:defect-chain-splitting}
\end{equation}
If $\theta=(\sigma,\tau)$ contributes to $H_k^{i,j}(n)$, then
\begin{equation}
 \operatorname{fd}_n(\theta)=\ell_n-k-j.
 \label{eq:defect-degree-relation}
\end{equation}
Consequently,
\begin{equation}
 H_k^{i,j}(n)=0
 \qquad\text{unless}\qquad
 k+j\leq\ell_n\leq i+j.
 \label{eq:homology-support-wedge}
\end{equation}
For every $d\geq0$,
\begin{align}
 \mathcal H_{n,d}(t,u)
 &=\sum_k\operatorname{ch}_{t,u}
       H_k(C_\bullet^{t,q}(\delta_n;d))
 \notag\\
 &=\sum_{\substack{k\geq1,\ j\geq0\\k+j=\ell_n-d}}
   u^j[q^j]\operatorname{ch}_{t,q}
   H_k(C_\bullet^{t,q}(\delta_n)),
 \label{eq:defect-layer-homology}\\
 \sum_k(-1)^{\ell_n-k}\operatorname{ch}_{t,-u}
       H_k(C_\bullet^{t,q}(\delta_n;d))
 &=(-1)^d\mathcal H_{n,d}(t,u).
 \label{eq:single-defect-euler}
\end{align}
Here $\operatorname{ch}_{t,u}$ and $\operatorname{ch}_{t,-u}$ mean
substitution of $q=u$ and $q=-u$, respectively.  Summing over $d$ gives
\begin{equation}
 \sum_{k=1}^{\ell_n}(-1)^{\ell_n-k}
 \operatorname{ch}_{t,-u}H_k(C_\bullet^{t,q}(\delta_n))
 =\sum_{d=0}^{\ell_n-1}(-1)^d\mathcal H_{n,d}(t,u).
 \label{eq:defect-euler-decomposition}
\end{equation}
If $\dim_n(f):=n![p_1^n]f$ denotes the dimension of the virtual
$\mathfrak S_n$-module with Frobenius characteristic $f$, then
\begin{equation}
 n!Z_{\delta_n}(t,-u)
 =\dim_n\!\left(\sum_{d=0}^{\ell_n-1}
                  (-1)^d\mathcal H_{n,d}(t,u)\right).
 \label{eq:defect-dimension-positivity}
\end{equation}
Thus the bivariate positivity theorem states that the scalar dimension of an
alternating sum of classes with nonnegative ribbon Schur expansions is coefficientwise
nonnegative.
\end{theorem}

\begin{proof}
For a total decoration $\theta=(\sigma,\tau)$, put
$k_\theta=|F_{\delta_n}(\theta)|+1$.  Factorization at these cuts partitions
the ordinary alphabet $[\ell_n]$ into $k_\theta$ nonempty intervals, each
invariant under $\sigma$.  Every interval contains at least one cycle, so
$c(\sigma)\geq k_\theta$ and
$\operatorname{fd}_n(\theta)=c(\sigma)-k_\theta\geq0$.  Also
$c(\sigma)\leq\ell_n$ and $k_\theta\geq1$, which gives the upper bound
$\operatorname{fd}_n(\theta)\leq\ell_n-1$.

Every nonzero face multiplies two adjacent local decorations, so it preserves
their iterated product and hence the total decoration.  The external
$\mathfrak S_n$-action changes only block labels and preserves it as well.
The spaces $C_\bullet^{t,q}(\delta_n;d)$ are therefore bigraded
$\mathfrak S_n$-subcomplexes, and the preceding bounds give
\eqref{eq:defect-chain-splitting}.  By
\cref{thm:total-decoration-ribbon-homology}, the fiber indexed by $\theta$
contributes in homological degree $k=k_\theta$.  Since
$j=\ell_n-c(\sigma)$, this proves \eqref{eq:defect-degree-relation} and the
first inequality in \eqref{eq:homology-support-wedge}.  The other inequality follows from
\[
 i+j=c(\sigma)+c(\tau)-1+\ell_n-c(\sigma)
     =\ell_n+c(\tau)-1\geq\ell_n.
\]

Formula \eqref{eq:total-decoration-homology-character}, grouped by the value
of $\operatorname{fd}_n$, gives \eqref{eq:defect-layer-homology} and also
shows that no negative-defect summand occurs.  A summand in homological degree
$k$ and $q$-degree $j$ has Euler sign
\[
 (-1)^{\ell_n-k}(-1)^j
 =(-1)^{\ell_n-k-j}
 =(-1)^{\operatorname{fd}_n(\theta)}.
\]
The sign is therefore constant on the homology of each defect subcomplex,
which proves \eqref{eq:single-defect-euler} and
\eqref{eq:defect-euler-decomposition}.  Applying $\dim_n$ and using
Euler--Poincar\'e together with \eqref{eq:bigraded-euler} gives
\eqref{eq:defect-dimension-positivity}.
\end{proof}

\begin{remark}[A non-Schur-positive virtual class]
The virtual Frobenius characteristic in
\eqref{eq:defect-euler-decomposition} need not be coefficientwise
Schur-positive.  For $n=4$, direct substitution in the exact-cut formula gives
\[
 \sum_{k=1}^{2}(-1)^{2-k}
 \operatorname{ch}_{t,-u}H_k(C_\bullet^{t,q}(\delta_4))
 =t^2(t+1)^2r_{(2,2)}
  +t(t+1)\bigl(u(t+2)-t\bigr)r_{(4)}.
\]
Since $r_{(4)}=s_{(4)}$ and
$r_{(2,2)}=s_{(3,1)}+s_{(2,2)}$, the coefficient of $s_{(4)}$ is
$t(t+1)(u(t+2)-t)$ and is not coefficientwise nonnegative.  Thus the virtual
class itself need not be coefficientwise Schur-positive, although its
dimension is coefficientwise nonnegative by
\cref{thm:cycle-sign-coherence}.
\end{remark}

The defect-zero layer is the edge $k+j=\ell_n$ of
\eqref{eq:homology-support-wedge}.  It has an explicit Frobenius formula.  If
$\lambda=(\lambda_1,\ldots,\lambda_r)$ is a composition, write
\[
 2\lambda=(2\lambda_1,\ldots,2\lambda_r),
 \qquad
 2\lambda-\mathbf e_r
 =(2\lambda_1,\ldots,2\lambda_{r-1},2\lambda_r-1).
\]

\begin{corollary}[The defect-zero edge]
\label[corollary]{cor:defect-zero-edge}
For $N\geq1$,
\begin{equation}
 \mathcal H_{2N,0}(t,u)
 =\sum_{\lambda\models N}
   u^{N-\ell(\lambda)}
   \left(\prod_{a=1}^{\ell(\lambda)}
     (\lambda_a-1)!(t)_{\lambda_a+1}\right)
   r_{2\lambda}.
 \label{eq:even-defect-zero-edge}
\end{equation}
For $N\geq0$,
\begin{equation}
 \mathcal H_{2N+1,0}(t,u)
 =\sum_{\lambda\models N+1}
   u^{N+1-\ell(\lambda)}
   \left(\prod_{a=1}^{\ell(\lambda)}(\lambda_a-1)!\right)
   \left(\prod_{a=1}^{\ell(\lambda)-1}(t)_{\lambda_a+1}\right)
   (t)_{\lambda_{\ell(\lambda)}}
   r_{2\lambda-\mathbf e_{\ell(\lambda)}}.
 \label{eq:odd-defect-zero-edge}
\end{equation}
Because the $u$-degree determines the homological degree on this edge,
these two identities determine every representation with $k+j=\ell_n$.
\end{corollary}

\begin{proof}
Fix an exact cut set of size $k-1$.  In the even case it determines a
composition $\lambda\models N$ of length $k$ and the ribbon composition
$2\lambda$.  The corresponding ordinary factor intervals have sizes
$\lambda_1,\ldots,\lambda_k$.  Defect zero means that the ordinary
permutation has exactly one cycle on each interval, giving
$(\lambda_a-1)!$ choices on the $a$-th interval.  A single cycle has no
proper invariant prefix, so these choices introduce no additional
factorization cuts.

The rooted factor on the $a$-th interval is arbitrary and has cycle
enumerator $(t)_{\lambda_a+1}$.  Since the ordinary component has $k$ cycles
and gluing the $k$ rooted factors decreases their total number of cycles by
$k-1$, the global $t$-weight is the product of these rooted cycle
enumerators.  The $q$-degree is $N-k$.  Summing over all exact cut sets proves
\eqref{eq:even-defect-zero-edge}.

In the odd case, an exact cut set determines a composition
$\lambda\models N+1$ of length $k$ and the ribbon composition
$2\lambda-\mathbf e_k$.  The ordinary interval sizes are again the
$\lambda_a$.  The first $k-1$ rooted intervals have sizes $\lambda_a$, while
the last has size $\lambda_k-1$.  Their cycle enumerators are respectively
$(t)_{\lambda_a+1}$ and $(t)_{\lambda_k}$, and the $q$-degree is $N+1-k$.
The same argument proves \eqref{eq:odd-defect-zero-edge}.
\end{proof}

The endpoint $k=1$ gives the maximal $u$-degree.  Applying the dimension map
to the single ribbon $r_{(n)}=h_n$ yields
\begin{align}
 [u^{N-1}](2N)!Z_{\delta_{2N}}(t,-u)
 &=(N-1)!(t)_{N+1} &&(N\geq1),
 \label{eq:even-extreme-u-coefficient}\\
 [u^N](2N+1)!Z_{\delta_{2N+1}}(t,-u)
 &=N!(t)_{N+1} &&(N\geq0).
 \label{eq:odd-extreme-u-coefficient}
\end{align}

\subsection{Logarithmic positivity of the even series}

Write
\[
 F_{\mathrm{even}}(x)
 :=\sum_{N\geq0}Z_{\delta_{2N}}(t,-u)x^N
 =B(x)^{-1}.
\]
The logarithmic derivative used in the positivity proof also controls the
logarithm of this series.

\begin{proposition}[Positive logarithmic coefficients]
\label{prop:connected-circular-deformation}
For $N\geq1$, define
\begin{equation}
 \Lambda_N(t,u)=\frac{(2N)!}{N}a_{N-1}(t,u),
 \label{eq:connected-circular-polynomial}
\end{equation}
where $-B'(x)/B(x)=\sum_{r\geq0}a_r(t,u)x^r$.  Then
\begin{equation}
 \Lambda_N(t,u)\in\mathbb N[t,u],
 \qquad
 F_{\mathrm{even}}(x)
 =\exp\!\left(\sum_{N\geq1}
       \Lambda_N(t,u)\frac{x^N}{(2N)!}\right).
 \label{eq:connected-even-exponential-formula}
\end{equation}
At $u=1$,
\begin{equation}
 \Lambda_N(t,1)=\frac{(2N)!}{N}\,\Omega(C_{2N}^{\circ};t),
 \label{eq:connected-circular-specialization}
\end{equation}
where $C_{2N}^{\circ}$ is the alternating circular fence with sign word
$(+,-)^N$, with $C_2^{\circ}$ understood as the two-element chain.
\end{proposition}

\begin{proof}
Since $F_{\mathrm{even}}=B^{-1}$, integration gives
\[
 \log F_{\mathrm{even}}(x)
 =\sum_{N\geq1}\frac{a_{N-1}(t,u)}{N}x^N,
\]
which proves the exponential formula.  By
\cref{lem:bivariate-log-derivative-positivity}, the polynomials
$\Lambda_N$ have nonnegative rational coefficients.

For integrality, put
\[
 M_{2r}(t,u)=(2r)!Z_{\delta_{2r}}(t,-u),\qquad
 M_{2r+1}(t,u)=0,\qquad M_0(t,u)=1.
\]
\Cref{thm:cycle-sign-coherence} gives $M_s\in\mathbb Z[t,u]$.  The
moment--cumulant formula applied to
\[
 F_{\mathrm{even}}(z^2)=\sum_{s\geq0}M_s(t,u)\frac{z^s}{s!}
\]
gives
\[
 \Lambda_N(t,u)=
 \sum_{\substack{\mathcal Q\text{ a set partition of }[2N]\\
                   |D|\text{ even for every }D\in\mathcal Q}}
 (-1)^{|\mathcal Q|-1}(|\mathcal Q|-1)!
 \prod_{D\in\mathcal Q}M_{|D|}(t,u).
\]
Thus $\Lambda_N\in\mathbb Z[t,u]$, and coefficientwise nonnegativity over
$\mathbb Q$ proves $\Lambda_N\in\mathbb N[t,u]$.

It remains to prove \eqref{eq:connected-circular-specialization}.  For a
positive integer $m$, let $K_m=[\min(i,j)]_{i,j\in[m]}$.  A principal minor
indexed by $1\leq s_1<\cdots<s_r\leq m$ equals
\[
 s_1(s_2-s_1)\cdots(s_r-s_{r-1}).
\]
With $d_1=s_1$, $d_i=s_i-s_{i-1}$ for $2\leq i\leq r$, and
$d_{r+1}=m+1-s_r$, the sum of these minors is
\[
 [z^{m+1}]
 \left(\frac{z}{(1-z)^2}\right)^r\frac{z}{1-z}
 =\binom{m+r}{2r}.
\]
It follows that
\[
 B(x)|_{t=m,u=1}=\det(I_m-xK_m),
\]
and hence
\[
 -\frac{B'(x)}{B(x)}\bigg|_{t=m,u=1}
 =\sum_{N\geq1}\operatorname{tr}(K_m^N)x^{N-1}.
\]
To evaluate the trace, assign values $y_1,\ldots,y_N\in[m]$ to the maximal
elements of $C_{2N}^{\circ}$.  The minimal element between $y_{i-1}$ and
$y_i$ has $\min(y_{i-1},y_i)$ possible values, with cyclic indices.  Therefore
\[
 \Omega(C_{2N}^{\circ};m)=\operatorname{tr}(K_m^N).
\]
Both sides are polynomials in $m$, so interpolation gives
$a_{N-1}(t,1)=\Omega(C_{2N}^{\circ};t)$ and proves
\eqref{eq:connected-circular-specialization}.
\end{proof}

\section{Conclusion}
\label{sec:conclusion}

Total permutation decorations provide a common structure for the homological
and enumerative results.  Their simultaneous factorization cuts identify the
classical ribbon complex in each fiber and are respected by the comonoid
cut coproduct on labelled complexes.  This yields all bigraded homology representations, cut-induced
injections at factorization cuts, and the identification of
cut primitives with $H_1$.

For zigzag compositions, the excess of ordinary cycles over factorization
intervals is the nonnegative defect
$\lceil n/2\rceil-k-j$.  It determines the homological support, fixes the Euler
sign on each defect layer, and leads to explicit Frobenius formulas on the
defect-zero edge.  The Riccati recurrences prove coefficientwise positivity
of the reflection-length refinement and its logarithmic coefficients.  The
$n=4$ virtual class shows that the scalar positivity does not generally lift
to coefficientwise Schur positivity before taking dimensions.

\section*{Acknowledgments}
The author thanks Professor Yanpeng Li for helpful discussions and comments on
the organization of the proofs.

\section*{Use of generative AI}
Generative-AI tools (OpenAI ChatGPT and Codex, August 2026) were used for
language editing, notation and cross-reference checks, bibliographic
verification, and drafting finite verification scripts.  No computer output
is used as a proof.  The author independently checked every mathematical
statement and assumes full responsibility for the manuscript.

\bibliographystyle{abbrvnat}
\bibliography{references_jca_focused}

@article{AguiarBergeronSottile,
  author  = {Aguiar, Marcelo and Bergeron, Nantel and Sottile, Frank},
  title   = {Combinatorial {H}opf algebras and generalized {D}ehn--{S}ommerville relations},
  journal = {Compositio Mathematica},
  volume  = {142},
  number  = {1},
  year    = {2006},
  pages   = {1--30},
  doi     = {10.1112/S0010437X0500165X}
}

@article{EhrenborgJung2013,
  author  = {Ehrenborg, Richard and Jung, JiYoon},
  title   = {The topology of restricted partition posets},
  journal = {Journal of Algebraic Combinatorics},
  volume  = {37},
  number  = {4},
  pages   = {643--666},
  year    = {2013},
  doi     = {10.1007/s10801-012-0379-8}
}

@misc{FerroniMoralesPanova,
  author        = {Ferroni, Luis and Morales, Alejandro H. and Panova, Greta},
  title         = {Skew shapes, {E}hrhart positivity and beyond},
  year          = {2025},
  eprint        = {2503.16403},
  archivePrefix = {arXiv},
  primaryClass  = {math.CO},
  note          = {To appear in Proceedings of the London Mathematical Society; arXiv:2503.16403v3},
  doi           = {10.48550/arXiv.2503.16403}
}

@article{GelfandEtAlNSym,
  author  = {Gelfand, Israel M. and Krob, Daniel and Lascoux, Alain and Leclerc, Bernard and Retakh, Vladimir S. and Thibon, Jean-Yves},
  title   = {Noncommutative symmetric functions},
  journal = {Advances in Mathematics},
  volume  = {112},
  number  = {2},
  pages   = {218--348},
  year    = {1995},
  doi     = {10.1006/aima.1995.1032}
}

@article{GesselReutenauer1993,
  author  = {Gessel, Ira M. and Reutenauer, Christophe},
  title   = {Counting permutations with given cycle structure and descent set},
  journal = {Journal of Combinatorial Theory, Series A},
  volume  = {64},
  number  = {2},
  pages   = {189--215},
  year    = {1993},
  doi     = {10.1016/0097-3165(93)90095-P}
}

@article{HivertNovelliThibon2008,
  author  = {Hivert, Florent and Novelli, Jean-Christophe and Thibon, Jean-Yves},
  title   = {Commutative combinatorial {H}opf algebras},
  journal = {Journal of Algebraic Combinatorics},
  volume  = {28},
  number  = {1},
  pages   = {65--95},
  year    = {2008},
  doi     = {10.1007/s10801-007-0077-0}
}

@misc{HershSundaram2026,
  author        = {Hersh, Patricia and Sundaram, Sheila},
  title         = {Stability and ribbon bases for the rank-selected homology of geometric lattices},
  year          = {2026},
  eprint        = {2604.06479},
  archivePrefix = {arXiv},
  primaryClass  = {math.CO},
  note          = {arXiv:2604.06479v2},
  doi           = {10.48550/arXiv.2604.06479}
}

@misc{KahaneFenceCoefficients,
  author        = {Kahane, Yakob},
  title         = {Combinatorial interpretation of the coefficients of the order polynomial of fence posets},
  year          = {2026},
  eprint        = {2607.11225},
  archivePrefix = {arXiv},
  primaryClass  = {math.CO},
  note          = {arXiv:2607.11225v1},
  doi           = {10.48550/arXiv.2607.11225}
}

@misc{HuangGreedyRecords2026,
  author        = {Huang, Pyuyi Chufeng},
  title         = {{B}ernstein Transfers and Greedy Records for Fence and Circular-Fence Order Polynomials},
  year          = {2026},
  eprint        = {2607.22767},
  archivePrefix = {arXiv},
  primaryClass  = {math.CO},
  note          = {arXiv:2607.22767v2},
  doi           = {10.48550/arXiv.2607.22767}
}

@article{KellerVandeBogert2025,
  author  = {VandeBogert, Keller},
  title   = {Ribbon {S}chur functors},
  journal = {Algebra \& Number Theory},
  volume  = {19},
  number  = {4},
  pages   = {771--834},
  year    = {2025},
  doi     = {10.2140/ant.2025.19.771}
}

@misc{AlmousaLu2026,
  author        = {Almousa, Ayah and Lu, Bryan},
  title         = {Ribbon complexes for the $0$-{H}ecke algebra},
  year          = {2026},
  eprint        = {2601.13324},
  archivePrefix = {arXiv},
  primaryClass  = {math.CO},
  note          = {arXiv:2601.13324v1},
  doi           = {10.48550/arXiv.2601.13324}
}

@article{Kreweras1965,
  author    = {Kreweras, Germain},
  title     = {Sur une classe de probl\`emes de d\'enombrement li\'es au treillis des partitions des entiers},
  journal   = {Cahiers du Bureau universitaire de recherche op\'erationnelle, S\'erie Recherche},
  volume    = {6},
  pages     = {9--107},
  year      = {1965},
  language  = {French},
  url       = {https://www.numdam.org/item/BURO_1965__6__9_0/}
}

@book{NISTHandbook2010,
  editor    = {Olver, Frank W. J. and Lozier, Daniel W. and Boisvert, Ronald F. and Clark, Charles W.},
  title     = {{NIST} Handbook of Mathematical Functions},
  publisher = {Cambridge University Press},
  address   = {Cambridge},
  year      = {2010},
  url       = {https://dlmf.nist.gov/15}
}

@book{StanleyEC1,
  author    = {Stanley, Richard P.},
  title     = {Enumerative Combinatorics, Volume 1},
  edition   = {Second},
  series    = {Cambridge Studies in Advanced Mathematics},
  volume    = {49},
  publisher = {Cambridge University Press},
  year      = {2012},
  doi       = {10.1017/CBO9781139058520}
}

@article{Stanley1982GroupsPosets,
  author  = {Stanley, Richard P.},
  title   = {Some aspects of groups acting on finite posets},
  journal = {Journal of Combinatorial Theory, Series A},
  volume  = {32},
  number  = {2},
  pages   = {132--161},
  year    = {1982},
  doi     = {10.1016/0097-3165(82)90017-6}
}

@incollection{Wachs2007,
  author    = {Wachs, Michelle L.},
  title     = {Poset topology: tools and applications},
  booktitle = {Geometric Combinatorics},
  series    = {IAS/Park City Mathematics Series},
  volume    = {13},
  pages     = {497--615},
  publisher = {American Mathematical Society},
  address   = {Providence, RI},
  year      = {2007},
  doi       = {10.1090/pcms/013/09}
}

@article{BergeronKrob1997,
  author  = {Bergeron, Fran\c{c}ois and Krob, Daniel},
  title   = {Acyclic Complexes Related to Noncommutative Symmetric Functions},
  journal = {Journal of Algebraic Combinatorics},
  volume  = {6},
  number  = {2},
  pages   = {103--117},
  year    = {1997},
  doi     = {10.1023/A:1008622519966}
}

@article{NovelliThibonToumazet2020,
  author  = {Novelli, Jean-Christophe and Thibon, Jean-Yves and Toumazet, Fr\'{e}d\'{e}ric},
  title   = {A Noncommutative Cycle Index and New Bases of Quasi-Symmetric Functions and Noncommutative Symmetric Functions},
  journal = {Annals of Combinatorics},
  volume  = {24},
  number  = {3},
  pages   = {557--576},
  year    = {2020},
  doi     = {10.1007/s00026-020-00504-5}
}

@book{AguiarMahajan2010,
  author    = {Aguiar, Marcelo and Mahajan, Swapneel},
  title     = {Monoidal Functors, Species and Hopf Algebras},
  series    = {CRM Monograph Series},
  volume    = {29},
  publisher = {American Mathematical Society},
  address   = {Providence, RI},
  year      = {2010},
  isbn      = {978-0-8218-4776-3},
  doi       = {10.1090/crmm/029}
}

\end{document}